\documentclass[a4paper,12pt]{amsart}
\usepackage[mathscr]{eucal}
\usepackage{amssymb}
\usepackage{latexsym}
\usepackage{enumerate}
\usepackage{amsthm}
\usepackage{amsmath}

\theoremstyle{plain}
\newtheorem{mthm}{Theorem}

\newtheorem{thm}{Theorem}[section]
\newtheorem{prop}[thm]{Proposition}
\newtheorem{lem}[thm]{Lemma}
\newtheorem{cor}[thm]{Corollary}

\newtheorem{que}[thm]{Question}
\theoremstyle{definition}
\newtheorem{defn}[thm]{Definition}

\newtheorem{rem}[thm]{Remark}

\evensidemargin\paperwidth
\advance\evensidemargin-\textwidth
\advance\oddsidemargin-1in
\evensidemargin\oddsidemargin

\topmargin\paperheight
\advance\topmargin-\textheight
\advance\topmargin-1in

\begin{document}
\title[superigidity into MCG]{Property $(TT)$ modulo $T$ and homomorphism superrigidity into mapping class groups}
\author{Masato MIMURA}
\address{Graduate School of Mathematical Sciences, 
University of Tokyo, Komaba, Tokyo, 153-8914, Japan}
\email{mimurac@ms.u-tokyo.ac.jp}


\thanks{The author is supported by JSPS Research Fellowships (PD) for Young Scientists No.23-247.}

\maketitle

\begin{abstract}
Every homomorphism from finite index subgroups of a universal lattices to mapping class groups of orientable surfaces (possibly with punctures), or to outer automorphism groups of finitely generated nonabelian free groups must have finite image. Here the universal lattice denotes the special linear group $G=\mathrm{SL}_m (\mathbb{Z}[x_1, \ldots , x_k])$ with $m$ at least $3$ and $k$ finite. Moreover, the same results hold ture if universal lattices are replaced with symplectic universal lattices $\mathrm{Sp}_{2m} (\mathbb{Z}[x_1, \ldots , x_k])$ with $m$ at least $2$. These results can be regarded as a non-arithmetization of the theorems of Farb--Kaimanovich--Masur and Bridson--Wade. A certain measure equivalence analogue is also established. To show the statements above, we introduce a notion of property $(\mathrm{TT})/\mathrm{T}$ (``$/\mathrm{T}$" stands for ``modulo trivial part"), which is a weakening of property $(\mathrm{TT})$ of N. Monod.  Furthermore, symplectic universal lattices $\mathrm{Sp}_{2m} (\mathbb{Z}[x_1, \ldots , x_k])$ with $m$ at least $3$ has the fixed point property for $L^p$-spaces for any $p\in (1,\infty)$.

\bigskip


\end{abstract}

\section{Introduction and main results}\label{sec:intro}
In this paper, we use the symbol $k$ for representing any finite nonnegative integer. For $g,l\geq 0$, let $\Sigma=\Sigma_{g,l}$ denote an oriented compact connected surface with $g$ genus and $l$ punctures, and set $\Sigma_g=\Sigma_{g,0}$. For $(\infty>) n\geq 2$, let $F_n$ denote a free group of rank $n$. We use $e$ as the identity in a group. Unless otherwise stated, every group is assumed to be countable and discrete. For a group $K$, $\lambda_K$ is the left regular representation of $K$ in $\ell^2(K)$.

The main topic of this paper is \textit{homomorphism superrigidity into mapping class groups $\mathrm{MCG}(\Sigma)$ of $($orientable$)$ surfaces and into outer automorphism groups $\mathrm{Out}(F_n)$ of $($finitely generated$)$ free groups} of certain groups $\Gamma$. By this terminology, we mean that every group homomorphism from $\Gamma$ into these groups has finite image. A breakthrough was done by B. Farb and H. Masur \cite{FaMas} in 1998 that states any irreducible higher rank lattice has homomorphism superrigidity into mapping class groups. They employs a study of the Poisson boundary of mapping class groups by A. V. Kaimanovich and Masur \cite{FaMas}, and therefore this theorem is called the Farb-Kaimanovich--Masur superrigidity. In 2010, M. R. Bridson and R. D. Wade \cite{BrWa} settled the question for $\mathrm{Out}(F_n)$ target case:

\begin{thm}$($\cite{FaMas}$;$\cite{BrWa}$)$\label{thm:break}
Let $G=\Pi _{i=1}^{m}\mathbf{G}_i(k_i)$, where $k_i$ are local field, $\mathbf{G}_i(k_i)$ are $k_i$-points of Zariski connected simple $k_i$-algebraic group $\mathbf{G}_i$. Assume $\sum_{i}\mathbf{G}_i(k_i) \geq 2$. Then for any irreducible lattice $\Gamma$ in $G$, every homomorphism from $\Gamma$ into $\mathrm{MCG}(\Sigma_{g})$ $(g\geq 0)$; and into $\mathrm{Out}(F_n)$ $(n\geq 2)$ have finite image.
\end{thm}
For their proofs, the \textit{Margulis Finiteness Property}, hereafter we call the MFP, plays a fundamental role. This property states that \textit{every normal subgroup of an irreducible higher rank lattice is either finite or of finite index.} Thanks to the MFP, the proofs are reduced to showing that every homomorphism above must have infinite kernel. Note that Margulis has also shown that these lattices are arithmetic. 

In this paper, we study some generalization of Theorem~\ref{thm:break} to matrix groups over \textit{general rings}. The special linear group $\mathrm{SL}_m(\mathbb{Z}[x_1,\ldots,x_k])$ with $m\geq 3$ over a finitely generated commutative polinomial ring over integers is called the \textit{universal lattice} by Y. Shalom \cite{Sha1}. Although this group itself is \textit{not} a lattice in semi-simple algebraic group (see Lemma~\ref{lem:Z-avnot}) whenever $k\geq 1$, it surjects onto higher rank lattices of the form of $\mathrm{SL}_m(\mathcal{O})$ (, if corresponding $k$ is large enough,) where $\mathcal{O}$ is an integer ring, such as $\mathrm{SL}_m(\mathbb{Z})$, $\mathrm{SL}_m(\mathbb{Z}[1/2])$, $\mathrm{SL}_m(\mathbb{Z}[\sqrt{2},\sqrt{3}])$, and $\mathrm{SL}_m(\mathbb{F}_q[t])$, where $\mathbb{F}_q$ denotes the finite field of order $q$ for $q$ a positive power of a prime. Universal lattices are expected to have certain properties which are common in the groups above, so that they can be seen \textit{mother groups} for these properties. For instance, Shalom \cite{Sha4} and L. Vaserstein \cite{Vas} have shown that universal lattices have \textit{property }$(\mathrm{T})$ of Kazhdan~\cite{Kaz}, see Sucsection~\ref{subsec:TT} and Remark~\ref{rem:Vas} for details. 

Now it is natural to ask whether universal lattices have homomorphism superrigidity into $\mathrm{MCG}(\Sigma)$ and $\mathrm{Out}(F_n)$. We also deal with the \textit{symplectic} version of this question. More precisely, with following Shalom, we call the symplectic group $\mathrm{Sp}_{2m}(\mathbb{Z}[x_1,\ldots,x_k])$ with $m\geq 2$ the \textit{symplectic universal lattice} and ask the same question to it. Note that Ershov--Jaikin-Zapirain--Kassabov \cite{EJK} have shown symplectic universal lattices have $(\mathrm{T})$, which plays a key role in this paper (see Section~\ref{sec:symp} and Section~\ref{sec:keyobs}). Our main result answers these question in the affirmative (note that Selberg lemma implies $\Gamma$ in the statement below can be taken as a torsion-free group, because universal and symplectic universal lattices can be embedded(\textit{not} discrete embedding) in special linear groups over $\mathbb{C}$):

\begin{mthm}\label{mthm:MCGOut1}
Let $\Gamma$ be a finite index subgroup either of $\mathrm{SL}_{m}(\mathbb{Z}[x_1,\ldots ,x_k])$, $m\geq 3$; or of $\mathrm{Sp}_{2m}(\mathbb{Z}[x_1,\ldots ,x_k])$, $m\geq 2$. Then for any $g\geq0$ and $n\geq 2$, every homomorphism 
$$
\Phi\colon \Gamma \to \mathrm{MCG}(\Sigma_g)
$$
and every homomorphism 
$$
\Psi \colon \Gamma \to \mathrm{Out}(F_n)
$$
have finite image. 
In particular, every homomorphism $\Gamma \to \mathrm{MCG}(\Sigma_{g,l}) \ \ \ (l\geq 1)$ 
and $ \Gamma \to \mathrm{Aut}(F_n)$ 
also have finite image.
\end{mthm}
Note that the latter cases follow from the fact that there are natural injections 
$$\mathrm{MCG}(\Sigma_{g,l}) \hookrightarrow \mathrm{Out}(F_{2g+l-1}) \textrm{ if $l\geq 1$,\ \ and \ }\mathrm{Aut}(F_n) \hookrightarrow \mathrm{Out}(F_{n+1})$$(the first injection is given via actions on the fundamental group of $\Sigma_{g,l}$).

The homomorphism superrigidity above has strong connection to group actions on low dimensional manifolds, which can be seen as a part of \textit{Zimmer's program}. Indeed, for $\mathrm{MCG}(\Sigma)$ target case, the result states that every $\Gamma$-action on a surface is finite \textit{up to isotopy}. For $\mathrm{Out}(F_n)$ target case, Farb and P. Shalen \cite{FaSh} have employed this superrigidity to deduce certain rigidity of a action of an irreducible higher rank lattice on a $3$-manifolds. Thus Theorem~\ref{mthm:MCGOut1} can be regarded as a mile stone to Zimmer's program for universal and symplectic universal lattices. As they are not realized as (arithmetic) lattices, we consider Theorem~\ref{mthm:MCGOut1} as a certain \textit{non-arithmetization} of Theorem~\ref{thm:break}. 

Let $A$ be a finitely generated and commutative ring (we always assume rings to be associative and unital). Since conclusion of Theorem~\ref{mthm:MCGOut1} remains true if $\Gamma$ is replaced with any group quotient of $\Gamma$, we obtain the same conclusion for $\Gamma$ being a finite index subgroup either of $\mathrm{E}_m(A)$, $m\geq 3$; or of $\mathrm{Ep}_{2m}(A)$, $m\geq 2$. Here $\mathrm{E}_m(A)$ and $\mathrm{Ep}_{2m}(A)$ respectively denote the \textit{elementary group} and the \textit{elementary symplectic group} over $A$. For the definitions of them, see Section~\ref{sec:symp}.

Let us mention some difficulty in establishing Theorem~\ref{mthm:MCGOut1}. As mentioned above, a key to proving Theorem~\ref{thm:break} is the MFP for irreducible higher rank lattices. More essentially, Bridson--Wade \cite{BrWa} have observed that if $\Gamma$ is $\mathbb{Z}\textit{-averse}$, then homomorphism superrigidity follows. Here a group $\Gamma$ is $\mathbb{Z}\textit{-averse}$ if for any finite index subgroup $\Gamma_0$ of $\Gamma$, \textit{no} normal subgroup $N$ of $\Gamma_0$ has a surjective homomorphism onto $\mathbb{Z}$. We, however, observe the following:
\begin{lem}\label{lem:Z-avnot} 
Universal lattices and symplectic universal lattices with $k\geq 1$ do $\mathrm{not}$ satisfy the MFP. Moreover, they are $\mathrm{not}$ $\mathbb{Z}\textit{-averse}$.
\end{lem}
We prove Lemma~\ref{lem:Z-avnot} (we set $k=1$ for simplicity) by considering the congruence kernel $N$ associated to the map $\mathbb{Z}[x]\twoheadrightarrow \mathbb{Z}; $$x\mapsto 0$. This $N$ is infinite and infinite index normal subgroup, and the map
$$
N \twoheadrightarrow \mathbb{Z};\ g \mapsto (g'\mid_{x=0})_{1,1}
$$
becomes a homomorphism. Here $'$ is a derivative: $\{x^n\}'=nx^{n-1}$ and $(\cdot)_{1,1}$ means the $(1,1)$-th entry.

There is another approach to show homomorphism superrigiity by M. Bestvina and K. Fujiwara \cite{BeFu}. They study second \textit{bounded cohomology} $H^2_{\mathrm{b}}(K;1_K,\mathbb{R})$, see Subsection~\ref{subsec:bdd}, of subgroups $K$ of mapping class groups with trivial coefficient and give an alternative proof of \cite{FaMas}, see also \S 3 of \cite{BrWa} for $\mathrm{Out}(F_n)$ target case. However, to the best of the author's knowledge, no results are known concerning second bounded cohomology of universal or symplectic universal lattices \textit{with trivial coefficient}, compare with Subsection~\ref{subsec:TT} . 

In this paper, we take the third approach, which is a modification of the second one above, to overcome these difficulties to show Theorem~\ref{mthm:MCGOut1}. Precisely, we study second bounded cohomology with unitary coefficients \textit{which does not contain the trivial representation}. Certain cases are studied by U. Hamenst\"{a}dt \cite{Ham} for $K\leqslant \mathrm{MCG}(\Sigma)$; and Bestvina--Bromberg--Fujiwara \cite{BBF} for $K\leqslant \mathrm{Out}(F_n)$, and these results are crucial for the proof of Theorem~\ref{mthm:MCGOut1}. We call a certain property $(\mathrm{TT})/\mathrm{T}$ (that means ``property $\boldsymbol{(\mathrm{TT})}$ modulo $\boldsymbol{\mathrm{T}}$rivial  linear part"), which relates to the cohomology above, see Subsection~\ref{subsec:TT}. 

In short, the proof of Theorem~\ref{mthm:MCGOut1} consists of the following two ingredients:
\begin{thm}\label{thm:TTMCG}
Let $\Gamma$ be a countable discrete subgroup. If $\Gamma$ has property $(\mathrm{TT})/\mathrm{T}$, in the sense in Definition~$\ref{defn:TTmT}$, then every homomorphism from $G$ into $\mathrm{MCG}(\Sigma_{g,l})$ $(g,l\geq 0)$; or into $\mathrm{Out}(F_n)$ $(n\geq 2)$ has finite image.
\end{thm}

\begin{thm}\label{thm:TTmodT}
The symplectic universal lattices $\mathrm{Sp}_{2m}(\mathbb{Z}[x_1,\ldots ,x_k])$, $m\geq 2$;,as well as the universal lattices $\mathrm{SL}_{m}(\mathbb{Z}[x_1,\ldots ,x_k])$, $m\geq 3$, have property $(\mathrm{TT})/\mathrm{T}$.
\end{thm}
The statement above for universal lattices has proven in the previous paper \cite[Remark~6.7]{Mim1} of the author.

We furthermore obtain some generalization of Theorem~\ref{mthm:MCGOut1}. Recall that Theorem~\ref{thm:break} is easily deduced if a lattice $\Gamma$ is noncocompact (because every solvable subgroup of $\mathrm{MCG}(\Sigma)$ or of $\mathrm{Out}(F_n)$ is virtually abelian, and because the MFP holds). For our case, it is much less trivial, but Theorem~\ref{mthm:MCGOut1} for universal lattice case can be shown in a simple argument, see Theorem~\ref{mthm:MCGOut2} below. The following theorem shows that homomorphism superridigity remains valid even in certain \textit{measure equivalent} setting (the author thanks Alex Furman for this formulation), in which the simple argument might not apply any more : 

\begin{mthm}\label{mthm:ME}
Let $A$ be a finitely generated, commutative, associative and unital ring. Let $\Gamma$ be either of $\mathrm{E}_{m}(A)$, $m\geq 3$; or of $\mathrm{Ep}_{2m}(A)$, $m\geq 2$. Suppose a countable group $\Lambda$ satisfies the following two conditions:
   \begin{enumerate}[$(i)$]
     \item $\Lambda$ is measure equivalent to $\Gamma$;
     \item For an ergodic ME-coupling $\Omega$ of $(\Gamma,\Lambda)$, there exists $\Lambda$-fundamental domain $X\simeq \Omega/\Lambda$ in  $\Omega$ such that the associated ME-cocycle $     \alpha \colon \Gamma \times X \to \Lambda$ 
     satisfies the $L^2$-condition. Namely, for any $\gamma\in \Gamma$, 
     $$
     |\alpha (\gamma, \cdot)|_{\Lambda} \in L^2(X),
     $$
   where $| \cdot |_{\Lambda}$ denotes a word metric on $\Lambda$ with respect to a finite generating set for $\Lambda$. 
   \end{enumerate}
Then every homomorphism from $\Lambda$ into $\mathrm{MCG}(\Sigma_{g,l})$ $(g,l\geq 0)$; or into $\mathrm{Out}(F_n)$ $(n\geq 2)$ has finite image. 

In particular, the following holds true: let $\Gamma$ be a finite index subgroup of $\mathrm{E}_{m}(A)$, $m\geq 3$; or of $\mathrm{Ep}_{2m}(A)$, $m\geq 2$. Suppose $G$ be a locally compact and second countable group and $G$ contains a group isomorphic to $\Gamma$ as a lattice. Then for any $\mathrm{cocompact}$ lattice $\Lambda$ in $G$, every homomorphism from $\Lambda$ into $\mathrm{MCG}(\Sigma_{g,l})$ $(g,l\geq 0)$; or into $\mathrm{Out}(F_n)$ $(n\geq 2)$ has finite image.
\end{mthm}
We refer to Section~\ref{sec:ME} for definitions and details on ME. Note that in the latter case, an ME-cocycle $\alpha \colon \Gamma \times G/\Lambda \to \Lambda$ satisfies $L^{\infty}$-condition for some choice of a $\Lambda$-fundamental domain, and that it is stronger than the $L^2$-condition in item $(ii)$. Also we mention that $\Lambda$ is always finitely generated because $\Gamma$ here has property $(\mathrm{T})$, see Subsection~\ref{subsec:TT} and Subsection~\ref{sebsec:sympT}, which implies finite generation; and because $(\mathrm{T})$ is an ME-invariant, see \cite{Fur1}.

In addition, we consider a generalization of properties $(\mathrm{T})$ and $(\mathrm{FH})$ of Bader--Furman--Gelander--Monod concerning general Banach space $B$, which they call $(\mathrm{T}_B)$ and $(\mathrm{F}_B)$. We note that for $(\infty >)p\gg 2$, property $(\mathrm{F}_{\mathcal{L}_p})$, which represents the fixed point property with respect to all affine isometric action on $L^p$ spaces, is strictly stronger than $(\mathrm{T})$ (, which is equivalent to $(\mathrm{F}_{\mathcal{L}_2})$ for locally compact second countable groups). See Theorem~\ref{thm:BFGM1} for details. We establish $(\mathrm{F}_{\mathcal{L}_p})$, $p\in (1,\infty)$ for symplectic universal lattices with $m\geq 3$:
\begin{thm}\label{thm:FFLp}
Let $p\in (1,\infty)$. Then symplectic universal lattices $\mathrm{Sp}_{2m}(\mathbb{Z}[x_1,\ldots ,x_k])$ for $m\geq 3$, as well as universal lattices $\mathrm{SL}_{m}(\mathbb{Z}[x_1,\ldots ,x_k])$ for $m\geq 4$, have property $(\mathrm{FF}_{\mathcal{L}_p})/\mathrm{T}$. In particular, these groups have property $(\mathrm{F}_{\mathcal{L}_p})$. Here for each $p$, $\mathcal{L}_p$ denotes the class of $L^p$ spaces.
\end{thm}
The statement above for universal lattices (for $m\geq 4$) is one of the main results in the previous paper \cite[Theorem~1.5]{Mim1} of the author.

Finally, as we mentioned above, Theorem~\ref{mthm:MCGOut1} can be deduced in considerably simpler argument for universal lattice case. This argument uses the \textit{distortion} of a finitely generated group, and we moreover show the following generalization.
\begin{mthm}\label{mthm:MCGOut2}
Let $\Gamma$ be a finite index subgroup of a noncommutative universal lattice $\mathrm{E}_{m}(\mathbb{Z}\langle x_1,\ldots ,x_k\rangle )$, $m\geq 3$. Then every homomorphism from $\Lambda$ into $\mathrm{MCG}(\Sigma_{g,l})$ $(g,l\geq 0)$; or into $\mathrm{Out}(F_n)$ $(n\geq 2)$ has finite image.
\end{mthm}
It follows that the same conclusion holds if $\Gamma$ is a finite index subgroup of $\mathrm{E}_m(R)$, $m\geq 3$, for any finitely generated ring $R$. However, we have no idea what happens if we consider $\Lambda$ in the setting in Theorem~\ref{mthm:ME} for $A=R$ noncommutative. Also, we warn that at the moment the proof of Theorem~\ref{mthm:MCGOut2} does \textit{not} work for symplectic universal lattices. See Remark~\ref{rem:spdame}.

\begin{rem}
Homomorphism superrigidity into mapping class groups can be combined with Roydon's theorem that the full isometry group of the Teichm\"{u}lar space os the mapping class group; and with Kerchoff's solution to the Nielson Realization Conjecture. Thus we have the following corollary: 
\begin{cor}\label{cor:fixed}
Let $\Gamma$ be either as in Theorem~\ref{mthm:MCGOut1}; as in Theorem~$\ref{mthm:MCGOut2}$; or as $\Lambda$ in Theorem~\ref{mthm:ME}. Then any isometric action of $\Gamma$ on a Teichm\"{u}ller space has a global fixed point.
\end{cor}
\end{rem}
\bigskip

\textbf{Organization of this paper}: Section~\ref{sec:prem} is for preliminaries of $\mathrm{MCG}(\Sigma)$, $\mathrm{Out}(F_n)$; property $(\mathrm{T})$; and bounded cohomology. The definition of property $(\mathrm{TT})/\mathrm{T}$ is recalled there. Section~\ref{sec:TTmT} is devoted to the proof of Theorem~\ref{thm:TTMCG}. In Section~\ref{sec:symp}, we treat fundamental results on symplectic universal lattices, with explaining some difference from universal lattices. In Section~\ref{sec:keyobs}, we prove a key theorem, Theorem~\ref{thm:short}, to establish Theorem~\ref{thm:TTmodT}. To deduce Theorem~\ref{thm:TTmodT} from Theorem~\ref{thm:short}, a certain property, which is written as $(*)$ there, shall play an imortant role. We thus obtain Theorem~\ref{mthm:MCGOut1} by combining Theorem~\ref{thm:TTMCG} and Theorem~\ref{thm:TTmodT}. In Section~\ref{sec:ME}, we treat measure equivalence and show Theorem~\ref{mthm:ME}.  Section~\ref{sec:tba} is employed to recall definitions of property $(\mathrm{T}_B)$ and $(\mathrm{F}_B)$, and to prove Theorem~\ref{thm:FFLp}. In Section~\ref{sec:short}, we introduce a shortcut argument for Theorem~\ref{mthm:MCGOut1} for universal lattices, which is based on distorted elements in finitely generated groups. There we will prove Theorem~\ref{mthm:MCGOut2}.

\section*{acknowledgments}
The author is indebted to his Ph.D advisor Narutaka Ozawa and posdoc advisor Masahiko Kanai for their suggestions. The symbol ``$(\mathrm{TT})/\mathrm{T}$" is suggested by Nicolas Monod, who he thanks. He is grateful to Cornelia Dru\c{t}u, Talia Fern\'{o}s, Ursula Hamenst\"{a}dt, Alex Lubotzky, and Andr\'{e}s Navas for fruitful conversations and references. He wishes to express his gratitude to Alex Furman; Andrei Jaikin-Zapirain and Martin Kassabov; and to Koji Fujiwara for letting him know their works in progress respectively.  Also thanks to Mladen Bestvina, Martin R. Bridson, Pierre de la Harpe, Benson Farb, David Fisher, Nigel Higson, Yoshikata Kida, Takefumi Kondo, Pierre Pansu, Guoliang Yu, and  Andrzej \.{Z}uk for their attentions and comments.

Most part of this work was done during a long-term stay (February, 2010--January, 2011) at EPFL (\'{E}cole Polytechnique F\'{e}d\'{e}rale de Lausanne), Switzerland of the author. He is truly grateful to Nicolas Monod and his secretary Marcia Gouffon for their warmhearted acceptence and hospitality. He also thanks Alain Valette for many helps and conversations in Switzerland. This stay was supported by  the Excellent Young Researcher Overseas Visiting Program by JSPS.

\section{Preliminaries}\label{sec:prem}
\subsection{Mapping class groups and automorphism groups of free groups}\label{subsec:MCG}
Let $\Sigma=\Sigma_{g,l}$ for $g,l\geq 0$. The \textit{mapping class group} of $\Sigma$ is defined as $\mathrm{MCG}(\Sigma):=\mathrm{Homeo}_{+}(\Sigma)/\mathrm{isotopy}$. Here a mapping class is allowed to permute punctures, and isotopies can rotate a neighborhood of a puncture. We say the surface $\Sigma=\Sigma_{g,l}$ is \textit{non-exceptional} if $3g+l\geq 5$ holds, and say it is \textit{exceptional} otherwise. For $n\geq 2$, $\mathrm{Aut}(F_n)$ denotes the \textit{automorphism group} of $F_n$ and $\mathrm{Out}(F_n)$ denotes the \textit{outer automorphism group} of $F_n$, namely, $\mathrm{Out}(F_n):=\mathrm{Aut}(F_n)/\mathrm{Inn}(F_n)$. 

It is a classical fact that $\mathrm{MCG}(\Sigma)$ is a finitely generated group  , and in fact it is finitely presented (see for instance Chapter 6 of a forthcoming book \cite{FaMar} of Farb--Margalit). Also, Nielsen has given finite presentations for $\mathrm{Out}(F_n)$ and $\mathrm{Aut}(F_n)$. For comprehensive treatment for $\mathrm{MCG}(\Sigma)$; and $\mathrm{Out}(F_n)$, we refer to a survay of N. V. Ivanov \cite{Iva} and \cite{FaMar}; and surveys of K. Vogtmann \cite{Vog}.

Via intersection numbers for $\Sigma_g$; and abelianization for $F_n$ respectively, we have the following surjections for $g\geq 1$ and $n\geq 2$:
$$
\mathrm{MCG}(\Sigma_g) \twoheadrightarrow \mathrm{Sp}_{2g}(\mathbb{Z});\ \ \mathrm{Out}(F_n) \twoheadrightarrow \mathrm{GL}_{n}(\mathbb{Z}).
$$
The \textit{Torelli groups} $\mathcal{T}_g \trianglelefteq \mathrm{MCG}(\Sigma_g)$; and $\overline{\mathrm{IA}}_n \trianglelefteq \mathrm{Out}(F_n)$ are defined as the kernel of the surjections above repsectively (the symbol $\mathrm{IA}$ stands for ``identity on the abelianization"). 

From this point, we collect some facts needed for the proof of Theorem~\ref{thm:TTMCG}. First we need element classification for $\mathrm{MCG}(\Sigma)$ and $\mathrm{Out}(F_n)$. The well-known classification theorem of Nielson--Thurston gives an element classification for $\mathrm{MCG}(\Sigma)$: each $f\in \mathrm{MCG}(\Sigma)$ is  either a torsion; \textit{reducible}; or \textit{pseudo-Anosov}. Here $f$ is said to be \textit{reducible} if it fixes some \textit{curve system}, a collection of isotopy classes of essential (not homotopic to one point nor homotopic to a boundary component) simple closed curves that are pairwise disjoint. A basic example of a reducible element is the Dehn twist along an essential simple closed curve.  A mapping class $f$ is said to \textit{pseudo-Anosov} if $f$ fixes exactly two points in the projective space $\mathcal{PMF}(\Sigma)$ of measured foliation on $\Sigma$ (see for instance \cite{FaMar} for details). In this case the action of $f$ on $\mathcal{PMF}(\Sigma)$ has a north-south dynamics. For two pseudo-Anosov element $f_1,f_2\in \mathrm{MCG}(\Sigma)$, we say they are \textit{independent} if the fixed point sets (in $\mathcal{PMF}(\Sigma)$) are disjoint. 

A similar concept to an pseudo-Anosov element (in $\mathrm{MCG}(\Sigma)$) is defined in $\mathrm{Out}(F_n)$ case, and it is called a \textit{fully irreducible element}: A subgroup $L\leqslant F_n$ is called a \textit{free factor} of $F_n$ if there exists $L'\leqslant F_n$ such that the free product $L\star L'$ is isomorphic to $F_n$. An element $f$ in $\mathrm{Out}(F_n)$ is said to be \textit{fully irreducible} if \textit{no} non-zero power of $f$ fixes the set of conjugacy classes of any free factor of $F_n$. A theorem of Levitt--Lustig \cite{LeLu} states that if $f$ is fully irreducible, then it acts on $\overline{\mathcal{PT}}$ with exactly two fixed points with north-south dynamics. Here $\overline{\mathcal{PT}}$ denotes Culler--Morgan's equivalent compactification $\overline{\mathcal{PT}}$ \cite{CuMo} of Vuller--Vogtmann's outer space \cite{CuVo}, and this can be seen an analogue of $\mathcal{PMF}(\Sigma)$ in $\mathrm{MCG}(\Sigma)$ case. Two fully irreducible elements are said to be \textit{independent} if their fixed point sets (in $\overline{\mathcal{PT}}$) are disjoint.

Secondly, we need subgroup classification for these groups. The following two theorems play roles:
\begin{thm}$($McCarthy--Papadopoulos \cite{McPa}$)$\label{thm:mcpa}
Let $\Sigma=\Sigma_{g,l}$ be a surface. Then each $K \leqslant \mathrm{MCG}(\Sigma)$ satisfies either of the following:
\begin{enumerate}[$(i)$]
  \item the group $K$ is finite$;$
  \item the group $K$ is $\mathrm{reducible}$. That means, there exists an $H$-preserved  curve system $\mathfrak{C}$ on $\Sigma;$
  \item the group $K$ has a pseudo-Anosov element, but there are no two independent pseudo-Anosov elements in $K$. In this case, $K$ is virtually $\mathbb{Z};$
  \item the group $K$ contains two independent pseudo-Anosov elements.
 \end{enumerate}
 \end{thm}

\begin{thm}$($Handel--Mosher \cite{HaMo}$)$\label{thm:hamo}
Let $n\geq 2$. Then each $K \leqslant \mathrm{Out}(F_n)$ satisfies either of the following:
\begin{enumerate}[$(i)$]
  \item the group $K$ is $\mathrm{not}$ fully irreducible, in the sense of Handel--Mosher. This means, for some finite index subgroup $K_0\leqslant K$, there exists a free factor $L \leqslant F_n$ with $L\ne \{e\}$ such that each element of $H_0$ preserves the set of conjugacy classes in $L;$ 
    \item the group $K$ has a fully irreducible element, but there are no two independent fully irreducible elements in $H$. In this case, $K$ is virtually $\mathbb{Z};$
  \item the group $K$ contains two independent fully irreducible elements. \end{enumerate}
 \end{thm}

Finally, we state the following theorem of Bass--Lubotzky \cite{BaLu} (another proof is given by Bridson--Wade \cite{BrWa}) on $\overline{\mathrm{IA}}_n$.
\begin{thm}$($Bass--Lubotzky \cite{BaLu}$)$\label{thm:BaLu}
Any nontrivial subgroup of $\overline{\mathrm{IA}}_n$ maps onto $\mathbb{Z}$.
\end{thm}

\subsection{Property $\boldsymbol{(\mathrm{T})}$, $\boldsymbol{(\mathrm{TT})}$, $\boldsymbol{(\mathrm{TT})/\mathrm{T}}$}\label{subsec:TT}
In this subsection, we only treat countable groups (in a similar way we can treat locally compact and second countable groups). Let $(\pi,\mathfrak{H})$ be a unitary representation of a group $G$. A map $b\colon G\to \mathfrak{H}$ is called a $\pi$\textit{-}($1$\textit{-})\textit{cocycle} if for any $g,h\in G$, 
$b(gh)=$$b(g)+\pi(g)b(h)$ holds (this condition is equivalent to that $\alpha (g)\cdot \xi$$:=\pi(g)\xi +b(g)$ becomes an affine isometric action of $G$ on $\mathfrak{H}$). A $\pi$-cocycle is called a $\pi$-\textit{coboundary} if there exists $\eta\in \mathfrak{H}$ such that for any $g\in G$, $b(g)=\eta -\pi(g)\eta$ holds (this means the associated affine isometric action has a global fixed point). We define the \textit{first group cohomology} as
$$
H^1(G;\pi):= \{\textrm{$\pi$-cocycles}\}/\{\textrm{$\pi$-coboundaries}\}.
$$

A map $b\colon G\to \mathfrak{H}$ is called a \textit{quasi-}$\pi$\textit{-}($1$\textit{-})\textit{cocycle} if 
$$
\sup_{g,h\in G}\| b(gh)- b(g)-\pi(g)b(h)\| <\infty.
$$
By definition, any linear combination of $\pi$-cocycles and bounded maps $G\to \mathfrak{H}$ is a quasi-$\pi$-cocycle. The following quotient vector space measures how many there are ``non-trivial" quasi-$\pi$-cocycles: 
$$
\widetilde{QH}(G;\pi):= \{\textrm{quasi-$\pi$-cocycles}\}/(\{\textrm{$\pi$-cocycles}\}+\{\textrm{bounded maps}\}). 
$$

Note that thanks to uniform convexity of Hilbert spaces, a cocycle is a coboundary if and only if it is bounded (see also Lemma 2.2.7 in \cite{BHV}).
\begin{defn}\label{def:T}
Let $G$ be a countable group.
\begin{enumerate}[$(1)$]
   \item A unitary $G$-representation $(\pi,\mathfrak{H})$ is said to \textit{weakly contain} $1_G$, written as $\pi \succeq 1_G$ if $\pi$ has \textit{almost invariant vectors}. Here $\pi$ is said to have \textit{almost invariant vectors} if for any $\epsilon >0$ and any finite subset $S\subseteq G$ there exists $\xi\in \mathfrak{H}$ such that 
   $$
   \max_{s\in S} \|\xi -\pi(s)\xi\| <\epsilon \|\xi \|.
   $$
   If $G$ is finitely generated, then one can fix $S$ in above as an  (arbitrarily taken) finite generating subset.
   \item (Kazhdan \cite{Kaz}) $G$ is said to have \textit{property} $(\mathrm{T})$ if whenever a unitary $G$-representation $\pi$ satisfies $\pi \succeq 1_G$, $\pi \supseteq 1_G$ (namely, $\mathfrak{H}^{\pi(G)}\ne 0$) holds. 
   \item Let $H$ be a subgroup of $G$. The pair $G\geqslant N$ is said to have \textit{relative property} $(\mathrm{T})$ if whenever a unitary $G$-representation $\pi$ satisfies $\pi \succeq 1_G$, $\pi\mid_H \supseteq 1_H$ (namely, $\mathfrak{H}^{\pi(H)}\ne 0$) holds.
   \item A group $G$ is said to have \textit{property} $(\mathrm{FH})$ if for any any unitary $G$-representation $\pi$, $H^1(G;\pi)=0$. This means that every cocycle into any unitary $G$-representation is a coboundary. This condition is equivalent to that every cocycle into any unitary $G$-representation is bounded (see the sentence above from this definition).
   \item (Monod \cite{Mon1}) $G$ is said to have \textit{property} $(\mathrm{TT})$ if every quasi-cocycle into any unitary $G$-representation is bounded. This is equivalent to the following two conditions:
   \begin{enumerate}[$(1)$]
     \item for any unitary $G$-representation $\pi$, $H^1(G;\pi)=0$;
     \item for any unitary $G$-representation $\pi$, $\widetilde{QH}(G;\pi)=0$.
   \end{enumerate}
   \item Let $U$ be a subset of $G$. We say the pair $G\supseteq U$ has \textit{relative property} $(\mathrm{TT})$ if every quasi-cocycle into any unitary $G$-representation is bounded on $U$.
   \end{enumerate}
\end{defn}

In the previous paper of the author \cite{Mim1}, a certain weaker version of property $(\mathrm{TT})$ is defined as follows:
\begin{defn}(\cite{Mim1})\label{defn:TTmT}
A countable group $G$ is said to have \textit{property} $(\mathrm{TT})/\mathrm{T}$ ($(\mathrm{TT})$ modulo $\mathrm{T}$) if for any unitary $G$-representation \textit{with} $\pi \not\supseteq 1_G$, every quasi-$\pi$-cocycle is bounded. This is equivalent to the following two conditions:
   \begin{enumerate}[$(a)$]
     \item for any unitary $G$-representation $\pi$ \textit{with} $\pi \not\supseteq 1_G$, $H^1(G;\pi)=0$;
     \item for any unitary $G$-representation $\pi$ \textit{with} $\pi \not\supseteq 1_G$, $\widetilde{QH}(G;\pi)=0$.
   \end{enumerate}
\end{defn}

The celebrated Delorme--Guichardet theorem asserts that property $(\mathrm{T})$ is equivalent to property $(\mathrm{FH})$ (for locally compact and second countable groups), see \S 2.12 in \cite{BHV}. In view of this, Monod called the ``quasification" of property $(\mathrm{FH})$ property $(\mathrm{TT})$. Property $(\mathrm{TT})$ relates bounded cohomology of groups (and this is the motivation of Burger and Monod to introduce this), see the next subsection.

Property $(\mathrm{T})$ represents strong rigidity of a group (for examples, see below). One counterpart of $(\mathrm{T})$ is \textit{Gromov's a-T-menability}, or the \textit{Haagerup property}, which states an existence of a \textit{metrically proper} cocycle into a unitary representation (recall a map $b\colon G\to \mathfrak{H}$ is said to be metrically proper if for any $M>0$, $\|b(g)\|>M$ for all $g\in G$ away from some compact subset). By definition, one has the following:
\begin{lem}\label{lem:TH}
Let $G$ and $H$ be countable groups. Suppose $G$ has $(\mathrm{T})$ and $H$ has the Haagerup property. Then every homomorphism from $G$ into $H$ has finite image.
\end{lem}
Moreover, note that $(\mathrm{T})$ implies \textit{finite generation} of the group.

We note that all amenable groups (including virtually abelian groups) and free groups (and virtually free groups) have the Haagerup property. Here we note that $(\mathrm{T})$($\Leftrightarrow (\mathrm{FH})$) passes to group quotients (which is trivial), and passes to finite index subgroups (for which we use induction). For property $(\mathrm{TT})/\mathrm{T}$, the same stability holds:
\begin{lem}\label{lem:sta}
Property $(\mathrm{TT})/\mathrm{T}$ passes to group quotients and to finite index subgroups.
\end{lem}
For the proof of heredity to finite index subgroups, we observe first that if $G\geqslant G_0$ and $G_0$-representation $\sigma \not\supseteq 1_{G_0}$, then the induction $\mathrm{Ind}_{G_0}^{G}\sigma \not\supseteq 1_G$. Secondly that from a finite index subgroup, one can induce quasi-cocycles without any problem. We note that it is \textit{not} known whether $(\mathrm{TT})/\mathrm{T}$ is an ME-invariant, see Section~\ref{sec:ME} for details.

A basic example of (infinite) groups with $(\mathrm{T})$ is a totally higher rank lattice, a lattice in a semi-simple algebraic group over local fields with each factor having local rank $\geq 2$, such as $\mathrm{SL}_{m}(\mathbb{Z})$, $m\geq 3$ and $\mathrm{Sp}_{2m}(\mathbb{Z}[1/2])$, $m\geq 2$ (a well-known example of relative property $(\mathrm{T})$ is $\mathrm{SL}_2(\mathbb{Z})\ltimes \mathbb{Z}^2 $$\geqslant \mathbb{Z}^2$, and $\mathrm{SL}_2(\mathbb{Z})\ltimes \mathbb{Z}^2$ itself does not have $(\mathrm{T})$). In fact, a totally higher rank lattice has property $(\mathrm{TT})$ \cite{BuMo}. It is known that a lattice in $\mathrm{Sp}_{m,1}$, $m\geq 2$ and some hyperbolic group (such as some random groups and $\mathrm{Sp}_{n,1}, n\geq 2$) have property $(\mathrm{T})$. However, any hyperbolic $H$ group \textit{fails} to have $(\mathrm{TT})/\mathrm{T}$ because Mineyev--Monod--Shalom \cite{MMS} have shown that $H$ satisfies $\widetilde{QH}(H;\lambda_H)\ne 0$ (in fact it is infinite dimensional) unless $H$ is virtually $\mathbb{Z}$. Here $\lambda_H$ denotes the left regular representation. As is stated as Theorem~\ref{thm:TTmodT}, universal lattices and symplectic universal lattices have $(\mathrm{TT})/\mathrm{T}$. However to the best knowledge of the author, it is not known whether they have $(\mathrm{TT})$ (this is the motivation to define $(\mathrm{TT})/\mathrm{T}$). The following is a special case of a lemma in \cite{Mim3}, but we state it here for the convenience:
\begin{lem}\label{lem:T}
For countable groups, $(\mathrm{TT})/\mathrm{T}$ implies $(\mathrm{T})$.
\end{lem}
\begin{proof}
Let $G$ be a group with $(\mathrm{TT})/\mathrm{T}$. Then we know that for any unitary representation $\pi \not\supseteq 1_G$, $H^1(G;\pi)=0$. To verify $(\mathrm{T})$, it suffices to show that the abelianization $H:=G/[G,G]$ is finite. 

Suppose $H$ happens to be infinite. Then since $H$ is infinite abelian, $\lambda_H \not\supseteq 1_H$ but $\lambda_H \succeq 1_H$ (the latter follows from amenability of $H$, see Appendix G in \cite{BHV}). These show that $H^1(H;\lambda_H)\ne 0$, see \S 3.a in \cite{BFGM}. By pulling-back through $G\twoheadrightarrow H$, we have a unitary $G$-representation $\pi \not\supseteq 1_G$ with $H^1(G;\pi)\ne 0$. This is a contradiction.
\end{proof}
To sum up, we have the following implications between $(\mathrm{T})$, $(\mathrm{TT})$, and $(\mathrm{TT})/\mathrm{T}$:
$$
(\mathrm{TT})\Longrightarrow (\mathrm{TT})/\mathrm{T} \ {}^{\Longrightarrow}_{\not\Longleftarrow} \ (\mathrm{T}).
$$

\begin{que}\label{que:TT}
Does property $(\mathrm{TT})/\mathrm{T}$ imply $(\mathrm{TT})$? In particular, do universal lattices and symplectic universal lattices have vanishing $\widetilde{QH}$ with $\mathrm{trivial}$ representation coefficients?
\end{que}
The positive answer will imply vanishing of $\widetilde{QH}(G;1_G,\mathbb{R})$ and vanishing of the \textit{stable commutator length} on $G$ being a universal lattices or symplectic universal lattices, see \cite{Cal} on these topics.  Ver partial result is obtained for universal lattices of degree $\geq 6$, see Remark~\ref{rem:TTq} for details.

\subsection{Bounded cohomology and non-vanishing results}\label{subsec:bdd}
We shortly state the relation between $(\mathrm{TT})$ (and $(\mathrm{TT})/\mathrm{T}$) and \textit{bounded cohomology}. We refer \cite{BuMo}, \cite{Mon1} and \cite{Mon2} for comprehensive treatment. For a countable group $G$ and a Hilbert (unitary) $G$-module $(\pi,\mathfrak{H})$, the \textit{bounded cohomology} $H^{\bullet}_{\mathrm{b}}(G;\pi,\mathfrak{H})=H^{\bullet}(G;\pi)$ is defined to be the cohomology of the (homogeneous) complex 
$$
0 \longrightarrow \ell^{\infty}(G;\mathfrak{H})^G \longrightarrow \ell^{\infty}(G^2;\mathfrak{H})^G \longrightarrow \ell^{\infty}(G^3;\mathfrak{H})^G \longrightarrow \cdots
$$
of bounded invariant functions with the usual coboundary map. This complex is the complex of bounded functions, which is a subcomplex of the standard (homogeneous) bar complex for the ordinary group cohomology. 
Therefore there is a natural homomorphism 
$$
\Psi^{\bullet}\colon H^{\bullet}_{\mathrm{b}}(G;\pi)\to H^{\bullet}(G;\pi),
$$
which is called the \textit{comparison map}. Note that $\Psi^{\bullet}$ is neither injective nor surjective in general. The space $\widetilde{QH}(G;\pi)$ relates to bounded cohomology by the following well-known lemma:
\begin{lem}\label{lem:comp}
In the setting of the paragraph above, the following map
$$
\widetilde{QH}(G;\pi) \to \mathrm{Ker}\Psi^2\{H^{2}_{\mathrm{b}}(G;\pi)\to H^{2}(G;\pi)\};\ [b] \mapsto [\delta b]_{\mathrm{b}}
$$
gives an isomorphism between vector spaces. Here $[\cdot]_{\mathrm{b}}$ means the bounded cohomology class, and $\delta$ is the coboundary map.
\end{lem}
Hence if $G$ has $(\mathrm{TT})$, then for any unitary $G$-representation $\pi$, $H^2_{\mathrm{b}}(G;\pi)$ naturally injects into $H^2(G;\pi)$.

In this paper, we concentrate ourselves to unitary representations $\pi$ with $\pi\not\supseteq 1_G$ in relation to $(\mathrm{TT})/\mathrm{T}$. Some notable non-vanishing results of $\widetilde{QH}$ (and hence of $H^2_{\mathrm{b}}$) for certain subgroups of $\mathrm{MCG}(\Sigma)$ and of $\mathrm{Out}(F_n)$ have been obtained respectively by U. Hamenst\"{a}dt and a forthcoming work of Bestvina--Bromberg--Fujiwara. These are keys to the proof of Theorem~\ref{thm:TTMCG}, and we state them. 

\begin{thm}$($Hamenst\"{a}dt \cite[Corollary B, Proposition $5.1$]{Ham}$)$\label{thm:ham}
Let $\Sigma=\Sigma_{g,l}$ be a non-exceptional $($, namely, $3g+l\geq 5$$)$ surface. Let $K \leqslant \mathrm{MCG}(\Sigma)$. Suppose $K$ contains two independent pseudo-Anosov elements. Then 
$$
\widetilde{QH}(K;\lambda_K,\ell^2(K))\ne 0,
$$
holds. In fact, it is infinite dimensional.
\end{thm}
\begin{thm}$($Bestvina--Bromberg--Fujiwara \cite{BBF}$)$\label{thm:BBF}
Let $n\geq 2$ and $K \leqslant \mathrm{Out}(F_n)$. Suppose $K$ contains two independent fully irreducible elements. Then 
$$
\widetilde{QH}(K;\lambda_K,\ell^2(K))\ne 0,
$$
holds. In fact, it is infinite dimensional.
\end{thm}
Hamenst\"{a}dt's proof of Theorem~\ref{thm:ham} uses the hyperbolicity of the curve graph $\mathcal{C}(\Sigma)$ of non-exceptional $\Sigma$ \cite{MaMi} and the fact that the action $\mathrm{MCG}(\Sigma) \curvearrowright \mathcal{C}(\Sigma)$ is acylindrical \cite{Bow}, for details see \cite{Ham}. The approach of Bestvina--Bromberg--Fujiwara is to examine actions on a quasi-tree (a metric space which is quasi-isometric to a tree, see \cite{Man}).

\section{Property $(\mathrm{TT})/\mathrm{T}$ implies homomorphism superrigidity}\label{sec:TTmT}
Now we are in position to prove Theorem~\ref{thm:TTMCG}. Before proceeding in the proof of it, we expalin idea by showing the following ``baby case" (the author thanks Takefumi Kondo for this suggestion):
\begin{prop}\label{pro:baby}
Suppose a group $\Gamma$ have $(\mathrm{TT})/\mathrm{T}$. Then for a hyperbolic group $H$, every homomorphism $\Gamma \to H$ has finite image.
\end{prop}
The proof of Proposition~\ref{pro:baby} goes as follows: let $K$ be the image of $\Gamma \to H$. Recall by Lemma~\ref{lem:sta} and Lemma~\ref{lem:T}, $K$ has $(\mathrm{TT})/\mathrm{T}$ and in particular $(\mathrm{T})$. We appeal to a subgroup classification for hyperbolic groups: if $K$ is elementary hyperbolic (, namely, virtually $\mathbb{Z}$), then $K$ must be finite thanks to $(\mathrm{T})$  for $K$ and to Lemma~\ref{lem:TH}. If not, then \cite{MMS} states that $\widetilde{QH}(K;\lambda_K)$ is non-zero (in fact infinite dimensional), but this contradicts $(\mathrm{TT})/\mathrm{T}$ for $K$. This ends our proof.

We note that in the following proof, the argument of $\mathrm{MCG}(\Sigma)$ target case is a standard argument in this field, and the argument of $\mathrm{Out}(F_n)$ target case is  based on the argument in the work in \cite{BrWa}.
\begin{proof}(Theorem~\ref{thm:TTMCG})
We deal with the first two statements of the theorem. Recall by Lemma~\ref{lem:sta} that $\Phi(\Gamma)$ and $\Psi(\Gamma)$ has $(\mathrm{TT})/\mathrm{T}$ (and in particular $(\mathrm{T})$ by Lemma~\ref{lem:T}).

\textbf{Case1. with target of mapping class groups}

Let $\Phi\colon \Gamma\to \mathrm{MCG}(\Sigma_{g})$ be a homomorphism. First, we note that for exceptional cases the conclusion holds by Lemma~\ref{lem:TH}. Therefore, hereafter, we assume every surface which appears in this proof is non-exceptional. 

Let $K\leqslant \mathrm{MCG}(S)$ be the image of $\Phi$. 
We employ Theorem~\ref{thm:mcpa}, the subgroup classification result of \cite{McPa}. For convenience we restate here: $K$ is either of the following forms:
\begin{enumerate}[$(i)$]
  \item the group $K$ is finite;
  \item the group $K$ is reducible: there exists a $K$-invarinat curve system $\mathfrak{C}$;
  \item the group $K$ is virtually $\mathbb{Z}$;
  \item the group $K$ contains two independent pseudo-Anosov elements.
\end{enumerate}

Then we appeal to Theorem~\ref{thm:ham} of Hamenst\"{a}dt, and exclude option $(iv)$. This part is the key in this proof. More precisely, $K$ has property $(\mathrm{TT})/\mathrm{T}$. However, if option $(iv)$ occurs, then by Theorem~\ref{thm:ham}, 
$$
\widetilde{QH}(K;\lambda_K)\ne 0.
$$
Since option $(iv)$ (or the condition above already) implies that $K$ is infinite. Therefore $\lambda_K \nsupseteq 1_{K}$. This contradicts $(\mathrm{TT})/\mathrm{T}$ for $K$.

Property $(\mathrm{T})$ excludes option $(iii)$ (Lemma~\ref{lem:TH}). Therefore  for the proof, it suffices to show that  $K$ must be virtually abelian in option $(ii)$. 

Suppose option $(ii)$ occurs. Take a maximal curve system $\mathfrak{C}$ preserved by $K$. Cut $\Sigma$ open along $\mathfrak{C}$ and replace each boundary circle of the resulting bordered surface with a puncture. Then we get a possibly disconnected surface $\Sigma'$. Let $\Sigma'_1 ,\ldots ,\Sigma'_n$ be connected components of $\Sigma'$. Then there is a homomorphism:
\begin{align*}
K \twoheadrightarrow K' \leqslant(\mathrm{MCG}(\Sigma'_1)\times \cdots \times \mathrm{MCG}(\Sigma'_n))\rtimes \mathfrak{S} \longrightarrow \mathfrak{S}.
\end{align*}
Here $\mathfrak{S}$ is a subgroup of $\mathfrak{S}_n$, the  symmetric group of degree $n$, and $\mathfrak{S}$ acts by permutations on mapping class groups of homeomorphic surfaces among $\Sigma'_1,\ldots ,\Sigma'_n$. And the second homomorphism is the projection to $\mathfrak{S}$. It  is known the kernel of the map $K \to K'$ is a free abelian group, generated by multiple Dehn twists associated to the curves of the curve system $\mathfrak{C}$ (for instance, see Chapter 4 of \cite{Iva}). Take the kernel $K'_0 \trianglelefteq K'$ of the map $K' \to \mathfrak{S}$. Then there is a natural map from $K'_0$ to each $\mathrm{MCG}(\Sigma'_i)$, which takes the $i$-th component. 

Note that by definition $K'_0$ is a finite index subgroup of $K'$ and hence has $(\mathrm{TT})/\mathrm{T}$. Therefore again Theorem~\ref{thm:ham} tells us that each image of $K'_0$ inside $\mathrm{MCG}(\Sigma'_i)$ is either finite or reducible. However it cannot be reducible. Indeed, if an image of $K'_0$ fixes some curve system on $\Sigma'_i$, then by translation by $\mathfrak{S}$ we have a curve system on $\Sigma'_1\cup\cdots \cup \Sigma'_n$ which is preserved by $K'$. This contradicts the maximality of $\mathfrak{C}$. Therefore $\Lambda'_0$ must be finite, and thus $\Lambda$ is virtually abelian. Property $(\mathrm{T})$ ends our proof.

\textbf{Case2. with target of outer automorphism groups}
Recall from Subsection~\ref{subsec:MCG} the definition of $\overline{\mathrm{IA}}_n \trianglelefteq \mathrm{Out}(F_n)$. First we note that the conclusion holds for $n=2$ because then $\mathrm{Out}(F_n)\cong \mathrm{GL}_2(\mathbb{Z})$ has the Haagerup property.

Let $K\leqslant \mathrm{Out}(F_n)$ be the image of $\Gamma$ by $\Psi$. Firstly, we appeal to Theorem~\ref{thm:hamo}, the classification of subgroups in $\mathrm{Out}(F_n)$ by Handel--Mosher \cite{HaMo}: a subgroup $K \leqslant \mathrm{Out}(F_n)$ is either of the following forms:
\begin{enumerate}[$(i)$]
  \item the group $K$ is $\mathrm{not}$ fully irreducible: there exists a finite index subgroup which preserves each conjugacy class of some proper free factor of $F_n$;
  \item the group $K$ s virtually $\mathbb{Z}$;
  \item the group $K$ contains two independent fully irreducible elements. \end{enumerate}

Secondly,  we appeal to Theorem~\ref{thm:BBF} of Bestvina--Bromberg--Fujiwara and exclude option $(ii)$. This is done in a similar manner to one in Case $1$. Property $(\mathrm{T})$ excludes  option $(i)$. We shall show that a finite index subgroup $K_0\leqslant K$ as in option $(2)$ must be finite, by induction on $n$. 

If $n=2$, then we have seen it. Suppose the assertion holds true for every natural number $<n$, and we will verify the case of $n$. We take $K_0\leqslant K$ and a  free factor $L$ (with $L'\leqslant F_n$ such that $L\star L' \cong F_n$) as in option $(ii)$. Then for any $[f]\in K_0$, one can choose an element $g_f \in F_n$ such that $f(L)={g_{f}}^{-1} L g_{f}$ holds. Then the map $L\ni h$$\mapsto g_f h {g_{f}}^{-1}\in L$ is an element of $\mathrm{Aut}(L)$, and the image in $\mathrm{Out}(L)$ is uniquely determined by $[f]\in K_0$. This induces a well-defined group homomorphism $K_0\to \mathrm{Out}(L)$. Likewise, the action on $F_n/\langle \langle L \rangle\rangle$ induces a homomorphism $K_0\to \mathrm{Out}(L')$. Because the ranks of $L$ and $L'$ are strictly less than $n$, the assumption of induction applies. Hence both images are finite. By taking abelianization of $L\star L'$, with respect to the union of basis for $L$ and $L'$ the action of $K_0$ is of the following form:
$$
 \left( 
\begin{array}{cc}
 G & 0 \\ 
 * &  G' 
\end{array}
\right),
$$
where both $G$ and $G'$ are finite. Hence the image of the homomorphism $K_0\to \mathrm{GL}_n(\mathbb{Z})$ is virtually abelian. Since $K_0$ is a finite index subgroup of $K$,  $K_0$ has $(\mathrm{TT})/\mathrm{T}$ and in particular has $(\mathrm{T})$. By combining these two, we have the image is in fact finite. Therefore, the kernel $K_0'$ of the map above is a finite index subgroup of $K_0$.

Note that $K_0'$ is in $\overline{\mathrm{IA}}_n$. Finally we appeal to Theorem~\ref{thm:BaLu} of Bass--Lubotzky and Bridson--Wade that every nontrivial subgroup in $\overline{\mathrm{IA}}_n$ surjects onto $\mathbb{Z}$. 
On the other hand, $K_0'$ is a finite index subgroup of $K_0$ and hence has property $(\mathrm{T})$. Therefore $K_0'$ must be trivial by Lemma~\ref{lem:TH}, and thus $K$ is finite. This ends our proof.

\end{proof}

\begin{que}$($This question has been asked by D. Fisher$)$ 
Is property $(\mathrm{T})$ enough to deduce the homomorphism superrigidity into $\mathrm{MCG}(\Sigma)$? 
\end{que}
There are two facts which support this question. First, J. E. Anderson \cite{And} asserts that mapping class groups do not have $(\mathrm{T})$. Secondly, S.-K. Yeung has shown that a lattice in $\mathrm{Sp}_{m,1}$ with $m\geq 2$ (which has $(\mathrm{T})$) has homomorphism superrigidity into mapping class groups. Concerning $\mathrm{Out}(F_n)$, it seems that analogous problems are widely open.

\section{Symplectic universal lattices}\label{sec:symp}
Before proceeding to symplectic case, let us shortly recall the definition of elementary groups and universal lattices. For $m\geq 2$ and a unital (associative) ring $R$ (possibly noncommutative), \textit{elementary matirx} in $\mathrm{M}_m(R)$ is a matrix of the form $E_{i,j}(r):=I_m+r e_{i,j}$ ($1\leq i,j\leq m$, $i\ne j$; $r\in R$), here $e_{i,j}$ is the matrix with $1$ in $(i,j)$-th entry and $0$ in the other entries. The \textit{elementary group}, denoted by $\mathrm{E}_m(R)$, is the multiplicative group in $\mathrm{M}_m(R)$ generated by elementary matrices. Observe the Steinberg commutator relation
$$
[E_{i,j}(r),E_{j,l}(s)]=E_{i,l}(rs)\ \ \ i\ne j\ne l; r,s\in R.
$$
Then it follows that if $m\geq 3$ and $R$ is a finitely generated ring with the generating set $\{r_1,\ldots ,r_k\}$, then $\mathrm{E}_m(R)$ is a finitely generated group with a finite generating set 
$$
\{ E_{i,j}(r_t): 1\leq i,j\leq m,i\ne j; 1\leq t\leq k\}.
$$
The elementary group for $R=\mathbb{Z}[x_1,\ldots ,x_k]$ with $m\geq 3$ is called the \textit{universal lattice}. The Suslin stability theorem \cite{Sus} states that in that case, $\mathrm{E}_m$ coincides with the whole group $\mathrm{SL}_m$.
\subsection{One realization for elementary symplectic groups}
Firstly, for $m\geq 1$, we take the $2m$ by $2m$ alternating matrix as
$$
L_{m}=
 \left( 
\begin{array}{ccc}
L & 0  &  0  \\
 0  &\ddots &0 \\
 0 & 0 &  L
\end{array}
\right)\in \mathrm{M}_{2m}.
$$

Here 
$$
L=
 \left( 
\begin{array}{cc}
0 & 1  \\
 -1 &  0 
\end{array}
\right) \in \mathrm{M}_2.
$$
This choice of alternating matrix is standard, but in connection to property $(\mathrm{T})$, this choice is not good. Hence this alternating matrix $L_m$ later shall be replaced with another one $J_m$ in this thesis. See the next subsection for details.

Let $A$ be a \textit{commutative} ring (commutativity is needed here because we consider the transpose). Then the \textit{symplectic group} over $A$ of degree $2m$ is the following group: 
$$
\mathrm{Sp}_{2m}(A)=\{ g\in \mathrm{M}_2m(A):{}^tgL_mg=L_m\}.
$$
Here ${}^tg$ is the transpose matrix of $g$. We use the following permutation symbol on $\mathbb{Z}$:$
(2i)':=2i-1 ;\ (2i-1)'= 2i \ \ \ (i\in \mathbb{N})$. 
For any $(i,j)$ with $1\leq i\leq 2m$, $1\leq i\leq 2m$ and $i\ne j$; and any $a\in A$, we define the \textit{elementary symplectic matrix} as follows:
$$
SE_{i,j}(a):= 
\left\{
\begin{array}{cc}
 I_2m+a e_{i,j} & \textrm{if $i=j'$;} \\
 I_2m+ae_{i,j}-(-1)^{i+j}ae_{j',i'} & \textrm{if $i\ne j'$}.
\end{array}
\right.
$$ 
Here $e_{i,j}$ is the matrix with $(i,j)$-th entry $1$ and the other entries $0$.
The \textit{elementary symplectic group} over $A$ of degree $2m$ is the group generated by all elementary symplectic matrices, and it is written as $\mathrm{Ep}_{2m}(A)$. A priori, $\mathrm{Ep}_{2m}(A)$ is a (possibly non-proper) subgroup of $\mathrm{Sp}_{2m}(A)$.

Then there are some basic commutator relations between elementary symplectic matrices. They are so many, and here we only state some important relations. For the all precise relations, we refer to Lemma 4.1 of the paper \cite{GMV} of Grunewald--Menniche--Vaserstein. 

\begin{lem}\label{lem:spcom}
Let $a,b\in A$ are any element. Then there are the following formulae:
\begin{enumerate}[$(i)$]
\item \begin{enumerate}[$(1)$]
 \item $[SE_{i,i'}(a),SE_{i',l}(b)]=SE_{i,l}(ab)
         SE_{l',l}((-1)^{i+l}ab^2) $,

    $\textrm{if }$ $i\ne l, i'\ne l$.
 \item $[SE_{i,i'}(a),SE_{k,i'}(b)]=SE_{i,i'}(2ab+(-1)^{i+k}ab^2) $,
         
         if  $i\ne k, i'\ne k$.
 \item $[SE_{i,i'}(a),SE_{k,i}(b)]=SE_{k,i'}(-ab)
         SE_{k,k'}(-(-1)^{i+k}ab^2) $,

         $\textrm{if }$  $i\ne k, i'\ne k$.
  \item $[SE_{i,i'}(a),SE_{i,l}(b)]=SE_{l',l}(-2ab-(-1)^{i+l}ab^2) $,

  $\textrm{if }$ $ i\ne l, i'\ne l$.
      \end{enumerate}
   \item \begin{enumerate}[$(1)$]
   \item          $[SE_{i,j}(a),SE_{j,l}(b)]=SE_{i,l}(ab)$,

         $\textrm{if }$ $i'\ne j, j'\ne l, i\ne l, i'\ne l$.
   \item 
         $[SE_{i,j}(a),SE_{j,l}(b)]=SE_{i,l}(2ab)$,

         $\textrm{if }$ $i'\ne j, j'\ne l, i\ne l, i'= l$.
    \item 
         $[SE_{i,j}(a),SE_{k,j'}(b)]=SE_{k,i'}((-1)^{i+j}ab)$,
         
         $\textrm{if }$  $i'\ne j, j'\ne k,  i'\ne k, i\ne k$.
    \item $[SE_{i,j}(a),SE_{k,j'}(b)]=SE_{i,i'}(2(-1)^{i+j}ab)$,

         $\textrm{if }$ $ i'\ne j, j'\ne k,  i'\ne k, i= k$.
    \end{enumerate}
\end{enumerate}
\end{lem}
These relation implies that if $m\geq 2$, then whenever $A$ is finitely generated (as a ring), $\mathrm{Ep}_{2m}(A)$ is a finitely generated group. However, we warn that \textit{we need some care on a finite generating set} of $\mathrm{Ep}_{2m}(A)$. If one looks Lemma~\ref{lem:spcom} carefully, then one will notice that some rules are \textit{different} between elementary symplectic matrices of the form $SE_{i,j}(a)$ $(i'\ne j)$;  and those of the form $SE_{i,i'}(a)$. Structure of the latter in terms of commutators is \textit{much more complicated} than that of the former. Indeed, let $m\geq 2$ and suppose $A$ is finitely generated and $\{s_1,\ldots ,s_k\}$ is a finite generating set for $A$ (as a ring). Then the following finite set is a generating set for $\mathrm{Ep}_{2m}(A)$:
\begin{align*}
S&:=\{ SE_{i,j}(\pm s_l) : 1\leq i,j \leq 2m, i\ne j ,i'\ne j; 0 \leq l \leq k\} \\
& \cup  \{ SE_{i,i'}(\pm s_1^{\epsilon_1}s_2^{\epsilon_2}\cdots s_k^{\epsilon_k}) : 1\leq i \leq 2m; \epsilon_1,\ldots , \epsilon_k \in \{0,1\} \}.
\end{align*}
Here we regard $s_0=1$ and $s_l^0=1$. 

Note that if $m\geq m_0$, then there is a natural inclusion $
\mathrm{Ep}_{2m_0}(A) \hookrightarrow \mathrm{Ep}_{2m}(A),
$ which sends to the left upper corner.
\subsection{Another realization}
The realization in the previous subsection is natural, but it is not suited for study of property $(\mathrm{T})$. For this purpose, \textit{in this paper we use the following realization}, which may look awkward on other point of view: set the following alternating matrix (, which is conjugate to $L_m$):
$$
J_m:=\left( 
\begin{array}{cc}
   0 & I_m \\ 
   -I_m & 0 
\end{array}
\right)\in \mathrm{M}_{2m}
$$

\begin{defn}\label{def:symp}
Let $m\geq 1$ and $A$ be a commutative ring. 
\begin{enumerate}[$(i)$]
 \item
 The \textit{symplectic group} over $A$ of degree $2m$, written as $\mathrm{Sp}_{2m}(A)$, is defined as the multiplicative group of symplectic matrices in matrix ring $\mathrm{M}_{2m}(A)$ associated with the alternating matrix $J_m$, namely, 
  $$
\mathrm{Sp}_{2m}(A):= \{g\in \mathrm{M}_{2m}(A): {}^{t}g J_m g= J_m \}.
$$
 \item Matrices of the following form are called \textit{elementary symplectic matrices}:
   \begin{enumerate}[$(1)$]
     \item For $1\leq i,j \leq m$ with $i\ne j$ and $a\in A$, define 
      \begin{align*}
         B_{i,j}(a):&=I_{2m}+a (e_{i,m+j}+e_{j,m+i}), \\ 
         C_{i,j}(a):&=I_{2m}+a(e_{m+j,i}+e_{m+i,j} )(={}^tB_{i,j}(a)), \\
         D_{i,j}(a):&=I_{2m}+ae_{i,j}-ae_{m+j,m+i}.
       \end{align*}
     \item For $1\leq i \leq m$ and $a\in A$, define
       \begin{align*}
         B_{i,i}(a):=I_{2m}+a e_{i,m+i}, C_{i,i}(a):=I_{2m}+ae_{m+i,i}(={}^tB_{i,i}(a)). 
       \end{align*}
       
  Namely, for $i\ne j$,
\begin{align*}
&B_{i,j}(a):=\left( 
\begin{array}{cc}
   I_m & a(e_{i,j}+e_{j,i}) \\ 
   0 & I_m 
\end{array}
\right) , 
C_{i,j}(a):=\left( 
\begin{array}{cc}
   I_m & 0 \\ 
   a(e_{i,j}+e_{j,i}) & I_m 
\end{array}
\right) , \\
& D_{i,j}(a):=\left( 
\begin{array}{cc}
   I_m +ae_{i,j}&0 \\ 
   0 & I_m-ae_{j,i}
 \end{array}
 \right) ; 
\end{align*}
and for $i=j$,
\begin{align*}
B_{i,i}(a):=\left( 
\begin{array}{cc}
   I_m & ae_{i,i} \\ 
   0 & I_m 
\end{array}
\right) , 
C_{i,i}(a):=\left( 
\begin{array}{cc}
   I_m & 0 \\ 
   ae_{i,i} & I_m 
\end{array}
\right) .
\end{align*}
   \end{enumerate}
    \item The \textit{elementary symplectic group} $\mathrm{Ep}_{2m}(A)$ over $A$ of degree $2m$ is the subgroup of $\mathrm{Sp}_{2m}(A)$ generated by all elementary symplectic matrices in the sense above.
\end{enumerate}
\end{defn}

Note that in item $(ii)$ above, the case of $i\ne j$ corresponds to that of $SE_{l,k}(a)$ $(l'\ne k)$; and the case of $i= j$ corresponds to that of $SE_{l,l'}(a)$  in the previous subsection. The indices are however permuted.

The following is mere interpretation of an observation in the previous subsection to this realization:

\begin{lem}\label{lem:spfg}
Let $m\geq 2$ and suppose $A$ is finitely generated and $\{s_1,\ldots ,s_k\}$ is a finite generating set for $A$ $($as a ring$)$. Then the following finite set is a generating set for $\mathrm{Ep}_{2m}(A)$:
\begin{align*}
S&:=\{ B_{i,j}(\pm s_l),C_{i,j}(\pm s_l),D_{i,j}(\pm s_l) : 1\leq i,j \leq m, i\ne j; 0 \leq l \leq k\} \\
& \cup  \{ B_{i,i}(\pm s_1^{\epsilon_1}s_2^{\epsilon_2}\cdots s_k^{\epsilon_k}), C_{i,i}(\pm s_1^{\epsilon_1}s_2^{\epsilon_2}\cdots s_k^{\epsilon_k}) : 1\leq i \leq m; \epsilon_1,\ldots , \epsilon_k \in \{0,1\} \}.
\end{align*}
Here we regard $s_0=1$ and $s_l^0=1$. 
\end{lem}

Here we state the stability theorem of Grunewald--Mennicke--Vaserstein for symplectic case.
\begin{thm}$($\cite{GMV}$)$
Let  $k\in \mathbb{N}$ and $A=\mathbb{Z}[x_1,\ldots,x_k]$. Then for any $m\geq 2$, $\mathrm{Ep}_{2m}(A)=\mathrm{Sp}_{2m}(A)$.
\end{thm}
\begin{defn}\label{def:sul}
Let $m\geq 2$. Take any $k\in \mathbb{N}$. 
A \textit{symplectic universal lattice} of degree $2m$ denotes a group $
\mathrm{Sp}_{2m}(\mathbb{Z}[x_1,\ldots ,x_k]) (=\mathrm{Ep}_{2m}(\mathbb{Z}[x_1,\ldots ,x_k]))$.
\end{defn}

Finally, we define the following identification for certain subgroups of $\mathrm{Sp}_{2m}(A)$ (or $\mathrm{Ep}_{2m}(A)$):

\begin{defn}\label{def:idensym}
Let $A$ be a commutative ring.
\begin{enumerate}[$(i)$]
 \item For $m\geq m_0\geq 2$, by $\mathrm{SL}_{m_0}(A)\leqslant \mathrm{Sp}_{2m}(A)$ (or, $\mathrm{SL}_{m_0}(A)\hookrightarrow \mathrm{Sp}_{2m}(A)$) we mean the inclusion is realized in the following way:
 $$
  \left\{ 
 \left( 
\begin{array}{cccc}
W & 0 &0 & 0\\
 0 &  I_{m-m_0} & 0 &0 \\
 0 & 0 & {}^t W^{-1} &0 \\
 0 & 0  & 0 & I_{m-m_0} 
\end{array}
\right)
 : W \in \mathrm{SL}_{m_0} (A)
 \right\}
 \ \leqslant \ \mathrm{Sp}_{2m}(A).
  $$ 
  \item Let $m\geq 2$. 
We denote by $S^{m*}(A^m)$ the additive group of all symmetric matrices in $\mathrm{M}_m(A)$. 
By $\mathrm{E}_{m}(A) \ltimes S^{m*}(A^m)$$\trianglerighteq S^{m*}(A^m)$, we identify these groups respectively with 
\begin{align*} 
& \left\{ 
 (W,v):=\left( 
\begin{array}{c|c}
W & v \\
\hline 
 0 &  {}^{t}(W^{-1}) 
\end{array}
\right)
 : W \in \mathrm{E}_{m} (A), \, v \in S^{m*}(A^m)
 \right\}
  \\ \trianglerighteq &
 \left\{ 
 \left( 
\begin{array}{c|c}
I_m & v \\
\hline 
 0 &  I_m 
\end{array}
\right)
 :  v \in S^{m*}(A^m)
 \right\}.
\end{align*} 
Thus the action of $\mathrm{E}_{m} (A)$ on $S^{m*}(A^m)$ is:
$$
(W,0)(I_m,v)(W^{-1},0)=(I_m,W v {}^tW) \ \ (W\in \mathrm{E}_{m} (A), v\in S^{m*}(A^m)).
$$

Here $\mathrm{E}_m(A)$ denotes the elementary group over $A$.
 \item 
Let $m\geq m_0 \geq 2$. Then by 
 $\mathrm{Sp}_{2m_0}(A)\leqslant \mathrm{Sp}_{2m}(A)$ (or, $\mathrm{Sp}_{2m_0}(A)\hookrightarrow \mathrm{Sp}_{2m}(A)$), we mean the inclusion is realized as 
   $$
  \left\{ 
 \left( 
\begin{array}{cccc}
X & 0 & Y &0 \\
 0 &  I_{m-m_0} & 0 &0 \\
 Z & 0 & W &0 \\
 0 & 0  & 0 & I_{m-m_0} 
\end{array}
\right)
 : \left( 
\begin{array}{cc}
X & Y \\
 Z & W 
\end{array}
\right) \in \mathrm{Sp}_{2m_0} (A)
 \right\}
 \ \leqslant \ \mathrm{Sp}_{2m}(A).
  $$
\end{enumerate}
\end{defn}

The advantage of this realization of $\mathrm{Sp}_{2m}(A)$ (, namely the choice of $J_m$) is that then it is easy to express the pair in item $(ii)$. This pair plays a central role in the study of property $(\mathrm{T})$, see the next subsection.

\subsection{Known results on property $(\mathrm{T})$}\label{sebsec:sympT}
The first step to $(\mathrm{T})$ for symplectic universal lattices has benn done by M. Neuhauser:
\begin{thm}$($Neuhauser, \cite[Theorem $3.3$]{Neu}$)$\label{thm:neu}
Let $A=\mathbb{Z}[x_1,\ldots ,x_k]$ for any integer $k\in \mathbb{Z}$. Then for $m\geq 2$, the pair $\mathrm{E}_{m}(A) \ltimes S^{m*}(A^m)$$\trianglerighteq S^{m*}(A^m)$, which is defined as in item $(ii)$ of Definition~$\ref{def:idensym}$, has relative $(\mathrm{T})$. 
\end{thm}
Despite this theorem, it seemed challenging to extend from this relative property (and resulting relative $(\mathrm{T})$ for the pair $\mathrm{Sp}_{2m}(A) $$\geqslant S^{m*}(A^m)$) to the full property $(\mathrm{T})$. For some difficulty, see Remark~\ref{rem:Vas} in below. Nevertheless, in 2011, Ershov--Jaikin-Zapirain--Kassabov suceeded in establishing $(\mathrm{T})$ with substantially generalizing results.

\begin{thm}$($Ershov--Jaikin-Zapirain--Kassabov \cite{EJK}$)$\label{thm:sympT}
Let $\Phi$ be a reduced irreducible classical root system of rank at least $2$ and $A$ be a finitely generated commutative ring. Then $\mathrm{St}_{\Phi}(A)$, the twisted Steinberg group over $A$, has property $(\mathrm{T})$. 

In particular, every symplectic universal lattice $\mathrm{Sp}_{2m}(\mathbb{Z}[x_1,\ldots ,x_k])$ $(m\geq 2)$ has property $(\mathrm{T})$.
\end{thm}
Note that there is a surjection from $\mathrm{St}_{\Phi}(A)$ onto $\mathrm{Ep}_{2m}(A)$ for $\Phi$ being the root system of type $C_m$, and that thus the latter statement is a part of the former assertion, which treats notably wide cases. For details of twisted Steinberg group, see \cite{EJK}.

\begin{rem}\label{rem:Vas}
A notion of \textit{bounded generation} is a powerful tool to establish $(\mathrm{T})$ \cite{Sha1}. Recall that for a group $G$, and subsets $(S_j)_{i\in J}$ of $G$ indexed by $J$ with $e\in S_j$, we say $(S_j)_{i\in J}$ \textit{boundedly generates} $G$ if there exists $N\in \mathbb{N}$ such that 
$$
G=\Big(\bigcup_{j\in J}S_j\Big)^N
$$
holds. Namely, if there exists $N\in \mathbb{N}$ such that any element in $G$ can be written as product of $N$ elements of $\bigcup_{j\in J}S_j$. We warn that in some other literature, bounded generation is used for a \textit{confined} situation as follows: $J$ is a finite set, and each $S_j$ is a cyclic subgroup of $G$. 

In \cite{Vas}, L. Vaserstein has proven the following remarkable bounded generations for universal lattices and symplectic universal lattices:
\begin{thm}$($Vaserstein \cite{Vas}$)$\label{thm:vassp}
Let $A_k=\mathbb{Z}[x_1,\ldots l_k]$ for any $k$. Then for any $m\geq 3$, $\mathrm{SL}_{m}(A_k)$ is boundedly generated by the set of elementary matrices and the subgroup $\mathrm{SL}_2(A_k)$, which sits in the left upper corner. 

Furthermore, for any $m\geq 2$, $\mathrm{Sp}_{2m}(A_k)$ is boundedly generated by the set of elementary symplectic matrices and the subgroup $\mathrm{Sp}_2(A_k)$. 

Here $\mathrm{Sp}_{2m}(A_k)$ is realized as in Definition~$\ref{def:symp}$, and $\mathrm{Sp}_2(A_k)$ is realized in $\mathrm{Sp}_{2m}(A_k)$ as in item $(iii)$ of Definition~$\ref{def:idensym}$, namely, 
  $$
  \mathrm{Sp}_2(A_k):=
  \left\{ 
 \left( 
\begin{array}{cc|cc}
a & 0 & b &0 \\
 0 &  I_{m-1} & 0 &0 \\
 \hline
 c & 0 & d &0 \\
 0 & 0  & 0 & I_{m-1} 
\end{array}
\right)
 : \left( 
\begin{array}{cc}
a & b \\
 c & d 
\end{array}
\right) \in \mathrm{Sp}_{2} (A_k)
 \right\}
 \ \leqslant \ \mathrm{Sp}_{2m}(A_k).
$$
\end{thm}
The following deep theorem, Shalom's machinery, \cite{Sha3} enables us to establish $(\mathrm{T})$ for universal lattices from some relative $(\mathrm{T})$ and Vaserstein's bounded generation above:
\begin{thm}$($Shalom's machinery, \cite{Sha3}$)$\label{thm:shalom}
Suppose a quadraple $(G,H,N_1,N_2)$, where $G$ is a finitely generated group with finite abelianization; and $H,N_1,N_2$ are subgroups in $G$, satisfies the following four conditions:
\begin{enumerate}[$(i)$]
   \item The group $G$ is generated by $N_1$ and $N_2$ together$;$
   \item The subgroup $H$ normalizes $N_1$ and $N_2;$
   \item The group $G$ is boundedly generated 
   by $H, N_1$, and $N_2.$
   \item For both $i\in \{ 1,2 \}$, $N_i \leqslant G$ has relative $(\mathrm{T});$
\end{enumerate}
Then $G$ has property $(\mathrm{T})$.
\end{thm}
Indeed, for universal lattice $G=\mathrm{SL}_m(A)$ ($A=\mathbb{Z}[x_1,\ldots ,x_k]$, $m\geq 3$, set 
$$
H= \left\{\left( \begin{array}{c|c}
* & 0 \\
\hline
0 & 1 
\end{array}\right) \right\} \cong \mathrm{SL}_{m-1}(A), \ 
N_1= \left\{\left( \begin{array}{c|c}
I_{m-1} & * \\
\hline
0 & 1 
\end{array}\right) \right\} \cong A^{m-1},\ 
N_2= {}^tN_1 =\cong A^{m-1}.
$$
Then conditions $(i)$ and $(ii)$ are trivial, and condition $(iv)$ follows from a well-known result \cite{Sha1}. Condition $(iii)$ is highly non-trivial, but Theorem~\ref{thm:vassp} assures it. There is a generalization of Shalom's machinery for general Banach space case (other than the  Hilbert space case), see Section~\ref{sec:tba} and \cite{Mim1} for details. 

Now we explain why Shalom's machinery does \textit{not} work for symplectic universal lattices. Consider for instance the case of $G=\mathrm{Sp}_{4}(A)$ ($A=\mathbb{Z}[x_1,\ldots x_k]$). Then in the view of Theorem~\ref{thm:neu}, the following triple $(H,N_1,N_2)$ is a standard candidate for pairs with the Shalom's machinery:
\begin{align*}
&H=\mathrm{SL}_{2} (A):=
 \left\{ \left( 
\begin{array}{c|c}
W &   0 \\
\hline
0 & {}^tW^{-1} \\
\end{array}
\right)
 : 
W \in \mathrm{SL}_{2} (A)
 \right\}, \\
 N_1&:=
 \left\{ 
 \left( 
\begin{array}{c|c}
I_2 & v \\
\hline 
 0 &  I_2 
\end{array}
\right)
 :  v \in S^{2*}(A^2)
 \right\}\cong S^{2*}(A^2), 
  N_2:={}^t N_1
\cong S^{2*}(A^2). 
\end{align*} 

The point here is the following: \textit{even though the group $\mathrm{Sp}_2(A)$ and $\mathrm{SL}_2(A)$ are isomorphic as abstract groups, their realization inside $\mathrm{Sp}_{2m}(A)$ are completely different. A group relating to Vaserstein's bounded generation is $\mathrm{Sp}_2(A)$; but a group relating to relative $(\mathrm{T})$ is $\mathrm{SL}_2(A)$}. 

Therefore to prove of property $(\mathrm{T})$ for symplectic universal lattices, one likely needs the different technologies in the work in \cite{EJK}.
\end{rem}

\section{Proof of $\boldsymbol{(\mathrm{TT})/\mathrm{T}}$ for smplectic universal lattices}\label{sec:keyobs}
\subsection{Key observation}\label{subsec:key}
The following is a key to proving Theorem~\ref{thm:TTmodT}. Although this looks similar to Shalom's machinery (Theorem~\ref{thm:shalom}), the proof of it is extremely less involved. 
\begin{thm}\label{thm:short}
Suppose a triple $(G,H,U)$, where $G$ is a countable group; $H\leqslant G$ is a subgroup; and $U\subseteq G$ is a $\mathrm{subset}$ satisfies the following five conditions:
   \begin{enumerate}[$(i)$]
     \item The set $U$ generates $G;$
     \item The set $U$ is invariant under the conjugation of elements in $H$. Namely, for any $h\in H$, $hUh^{-1}\subseteq U;$
     \item The group $G$ is boundedly generated by $P$ and $H$ $($recall the definition from Remark~$\ref{rem:Vas}$$);$
          \item The pair $G\supseteq U$ has relative $(\mathrm{TT});$
     \item The group $G$ has $(\mathrm{T})$.
    \end{enumerate}
Then $G$ has property $(\mathrm{TT})/\mathrm{T}$.
\end{thm}

\begin{proof}
Let $\pi$ be a unitary $G$-representation with $\pi \not\supseteq 1_G$, and $b$ be a quasi-$\pi$-cocycle. By the definition of quasi-cocycles and condition $(iv)$, one can set \textit{finite} numbers 
$$
C_1:=\sup_{g_1,g_2\in G}\| b(g_1g_2)-b(g_1)-\rho(g_1)b(g_2)\|
,\ 
C_2:=\sup_{u\in U}\|b(u)\|. 
$$
Set $C=\max\{C_1,C_2\}<\infty$. 

We will show $b(H)$ is bounded. Take any $h\in H$ and any $u\in U$. Then we have the following inequalities:
\begin{align*}
&\|\pi(u)b(h)-b(h)\| \leq \|b(uh)-b(u)-b(h)\|+C \\
  \leq & \| b(uh)-b(h)\| +2C  =  \|b(h(h^{-1}uh))-b(h)\|+2C \\
  \leq & \| b(h)+\pi(h)b(h^{-1}uh)-b(h)\|+3C \\
  \leq & \|b(h^{-1}uh)\|+3C  \leq 4C.
\end{align*}
Here in the last line we use condition $(ii)$ and condition $(iv)$. Note that the last dominating term $4C$ is \textit{independent} of the choices of $h\in H$ and $u\in U$.

Now suppose $b(H)$ is not bounded. Then by the inequalities above, this means for any $\epsilon>0$, $\pi$ admits a unit vector $\xi$ which satisfies 
$$
\sup_{u\in U}\| \pi (u)\xi -\xi \| < \epsilon.
$$ 
Since $U$ is a generating set for $G$ (condition $(i)$), this means that $\pi \succeq 1_G$. However this contradicts property $(\mathrm{T})$ for $G$ (condition $(v)$). Therefore $b(H)$ is bounded, as claimed.

Finally, we obtain that $b(G)$ is bounded by employing bounded generation (condition $(iii)$).
\end{proof}
The reason why the proof of this theorem is simple is that the assumption requires property $(\mathrm{T})$ for the whole group, which itself is a big issue to check. Nevertheless, Theorem~\ref{thm:short} can apply to remarkable cases, as we will see in the next subsection. Furthermore, Theorem~\ref{thm:short} becomes more interesting if we consider analogues of $(\mathrm{T})$ and $(\mathrm{FH})$ in Banach space setting (recall from Subsection~\ref{subsec:TT} that $(\mathrm{TT})$ is a strengthening of $(\mathrm{FH})$, and not of $(\mathrm{T})$). We refer to Section~\ref{sec:tba} for details.

\subsection{Proof of Theorem~$\ref{thm:TTmodT}$}\label{subsec:TTmodT}
We need examples of triples $(G,H,U)$ for which Theorem~\ref{thm:short} applies. The following theorem due to N. Ozawa \cite{Oza} provides us with various of examples of group pairs with relative property $(\mathrm{TT})$, which relates to condition $(iv)$:
\begin{thm}$($Ozawa \cite[Proposition 3]{Oza}$)$\label{thm:relTTT}
Let $G=G_0\ltimes A$ be the semidirect product of a $\mathrm{abelian}$ group by a continuous action of $G_0$ $($$G_0$ and $A$ are locally compact second countable groups$)$. Then the following are equivalent:
  \begin{enumerate}[$(i)$]
  \item  The pair $G\geqslant A$ has relative $(\mathrm{T})$.
  \item  The pair $G\geqslant  A$ has relative $(\mathrm{TT})$.
  \end{enumerate}
\end{thm}
In fact, he has shown that under the condition above, relative $(\mathrm{T})$ for $G\geqslant  A$ is equivalent to ``relative property $(\mathrm{TTT})$" for $G\geqslant A$, which looks much stronger than relative property $(\mathrm{TT})$. See \cite{Oza} for details.

Thanks to Ozawa's theorem above, we have the following remarkable examples of Theorem~\ref{thm:short}:
\begin{thm}\label{thm:ex}
Let $A=\mathbb{Z}[x_1,\ldots ,x_k]$. Then the following triples $(G,H,U)$ satisfy the conditions in Theorem~\ref{thm:short}:
 \begin{enumerate}[$(1)$]
   \item $G=\mathrm{SL}_m(A)$ with $m\geq 3;$ $H=G;$ and $U=\bigcup_{g\in G}g\mathcal{E}g^{-1}$.
   \item $G=\mathrm{Sp}_{2m}(A)$ with $m\geq 2;$ $H=G;$ and $U=\bigcup_{g\in G}g\mathcal{S}g^{-1}$.
 \end{enumerate}
 Here $\mathcal{E}\subseteq \mathrm{SL}_m(A)$ is the set of all elementary matrices, and $\mathcal{S}\subseteq \mathrm{Sp}_{2m}(A)$ is the set of all elementary symplectic matrices.
\end{thm}
Theorem~\ref{thm:ex} together with Theorem~\ref{thm:short} immediately implies Theorem~\ref{thm:TTmodT}.

\begin{proof}(\textit{Theorem~\ref{thm:ex}})
First observe that we only need to confirm condition $(iv)$ (relative $(\mathrm{TT})$). Indeed, conditions $(i)$ and $(ii)$ are trivially satisfied; condition $(iii)$ also immediately follows (because $H$ itself is $G$!); and condition $(v)$ follows from a deep result of Shalom--Vaserstein (for universal lattices, see Remark~\ref{rem:Vas}), and of Ershov--Jaikin-Zapirain--Kassabov, Theorem~\ref{thm:sympT} (for symplectic universal lattices). Also note that the pair $G\supseteq \mathcal{S}$ for item $(2)$ has relative $(\mathrm{TT})$. This follows from an easy fact that $\mathrm{Sp}_{2m}(A)$ contains many copy of $\mathrm{E}_{m}(A) \ltimes S^{m*}(A^m)$$\trianglerighteq S^{m*}(A^m)$, together with Theorem~\ref{thm:neu} and Theorem~\ref{thm:relTTT}. For a similar reason, the pair $G\supseteq \mathcal{E}$ for item $(1)$ also has relative $(\mathrm{TT})$. 

Therefore to prove Theorem~\ref{thm:ex}, it suffices to check that the pair $(H,U_0)$ has the following property, here $U_0=\mathcal{E}$ for item $(1)$ and $U_0=\mathcal{S}$ for item $(2)$:
$$
(*):\left\{
\begin{array}{c}
\textrm{there exists a global bound }N\in \mathbb{N} \\
\textrm{ such that for any }g\in H\textrm{ and for any }s\in U_0\textrm{, }\\
gsg^{-1}\textrm{ can be writen as a product of at most }N \textrm{ elements of }U_0.
\end{array}\right.
$$

For both cases of $(1)$ and $(2)$, it is non-trivial to verify $(*)$. However, by clever matrix computation, it has been verified (implicitly) respectively by Park and Woodburn for $(1)$; and by V. I. Kopeiko for $(2)$:
\begin{thm}$($\cite{PaWo}$;$ \cite{Kop}$)$\label{thm:bddgen}
\begin{enumerate}[$(1)$]
 \item $($Park--Woodburn$)$ The pair $(H,U_0)=(\mathrm{SL}_{m}(A),\mathcal{E})$, $m\geq 3$, satisfies $(*)$.
 \item $($Kopeiko$)$ The pair $(H,U_0)=(\mathrm{Sp}_{2m}(A),\mathcal{S})$, $m\geq 2$, satisfies $(*)$.
 \end{enumerate}
\end{thm}
Here we just sketch the proof of item $(1)$ of Park--Woodburn. Take $s=E_{i,j}(a)\in \mathcal{S}$. Then for $g\in H$, 
$$
gsg^{-1}= I_m+(i\textrm{-th column vector of }g)\cdot a\cdot (j\textrm{-th column vector of }g^{-1}).
$$
Let $\mathbf{v}={}^t(v_1,\ldots ,v_m)$ be the $i$-th column vector of $g$, $\mathbf{w}=(w_1,\ldots ,w_m)$ be $a$ times the $j$-th row vector of $g^{-1}$, and $(g_1,\ldots ,g_m)$ be the $i$-th row vector of $g^{-1}$. Then since $i\ne j$, 
$
\sum_{l=1}^{m}g_lv_l=1$ and $\sum_{l=1}^{m}w_lv_l=0$. 
Therefore by letting $b_{l,n}=w_lg_n -w_n g_l$, one has 
$
\mathbf{w}=\sum_{l<n}b_{l,n}(v_n\mathbf{e}_l-v_l\mathbf{e}_n)
$, 
and thus (the following are due to Suslin),
\begin{align*}
 &gsg^{-1}=I_m + \mathbf{v}\cdot \mathbf{w} = I_m + \mathbf{v}\cdot \Big( \sum_{l<n}b_{ln}(v_n\mathbf{e}_l-v_l\mathbf{e}_n) \Big) \\
 =& I_m + \sum_{l<n} \mathbf{v}\cdot b_{l,n}(v_n\mathbf{e}_l-v_l\mathbf{e}_n) 
 = \prod_{l<n} (I_m + \mathbf{v}\cdot b_{l,n}(v_n\mathbf{e}_l-v_l\mathbf{e}_n)).
\end{align*}
Mennicke has shown that each factor in the very below side of the equalities above can be written as a product of bounded numbers of (only depend on $m$ and does not depend on $g$ and $s$), precisely at most $8+2(m-2)=2m+4$, elementary matrices. Therefore, for any $g\in H$ and $s\in U_0$ $gsg^{-1}$ is a product of at most $(m+2)m(m-1)$ elements in $U_0$, and this bound is independent of the choices of $g$ and $s$. For details, consult Lemma 2.3, Lemma 2.6, and Corollary 2.7 in \cite{PaWo}. 

The proof of item $(2)$ of Theorem~\ref{thm:bddgen} is much more complicated, but it is still done by direct matrix computation. See \S 1 of the paper \cite{Kop} of Kopeiko. 

Now Theorem~\ref{thm:bddgen} and thus condition $(iii)$ in Theorem~\ref{thm:short} has been verified in both cases. This ends our proof.
\end{proof}

\begin{rem}\label{rem:bdd}
If one looks at the proof above carefully, one may notice that property $(*)$ has potential for \textit{avioding} a deep result of bounded generation (condition $(iii)$ in Theorem~\ref{thm:short}). Instead, one needs to show another ``bounded generation" $(*)$ to check condition $(iv)$. However as we have seen in above, a proof of $(*)$ may be done by some elementary (but clever) computation, without deep tools. 

In view of this, it is natural to ask the following questions:
\begin{que}\label{que:chev}
 \begin{enumerate}[$(a)$]
    \item Is property $(*)$ is satisfied for the pair 
    $(H,U_0)=(\mathrm{E}_m(R),\mathcal{E})$, $m\geq 3$? Here 
    $R$ is the noncommutative polynomial ring $\mathbb{Z}\langle x_1,\ldots ,x_k\rangle$ and $\mathcal{E}$ is the set of all elementary matrices.
    \item For a reduced irreducible classical root system $\Phi$ and a commutative unital ring $A$, one can construct the \textit{elementary Chevalley group} $\mathrm{E}_{\Phi}(A)$ as the multiplicative group generated by the root subgroups with respect to the standard torus. Let $\mathcal{E}$ be the set of all elementary matrices with respect to $\Phi$. The question is the following: \textit{for $\Phi$ of rank $\geq 2$, does the pair $(H,U_0)=(\mathrm{E}_{\Phi}(A),\mathcal{E})$ satisfy $(*)$?}
    \end{enumerate}
\end{que}
Note that property $(\mathrm{T})$ for the groups above has been shown, respectively by Ershov--Jaikin-Zapirain \cite{ErJa} for $(a)$; and by  Ershov--Jaikin-Zapirain--Kassabov \cite{EJK} for $(b)$. Therefore positive answers will imply property $(\mathrm{TT})/\mathrm{T}$ and thus also homomorphism superridigity into $\mathrm{MCG}(\Sigma)$ and into $\mathrm{Out}(F_n)$. Moreover, at least in case $(a)$, positive answer to Question~\ref{que:chev} will imply property $(\mathrm{FF}_{\mathcal{L}_p})/\mathrm{T}$ (and in particular $(\mathrm{F}_{\mathcal{L}_p})$) for noncommutative universal lattices with $m\geq 4$. For details of this, see Section~\ref{sec:tba}. 

There is a work of G. Taddei which shows the normality of $\mathrm{E}_{\Phi}(A)$ in the Chevalley group. However this proof does not verify the existence of such a global bound  in $(*)$. For noncommutative ring, the argument in \cite{PaWo} does not work. 
Finally we note that bounded generation of the whole group, instead of checking just $(*)$, seems to be an extremely challenging problem.
\end{rem}

\begin{rem}
If one appeals to Vaserestin's bounded generation (Theorem~\ref{thm:vassp}), then one can get triples $(G,H,U)$ in Theorem~\ref{thm:short} with small $H$. For instance, for $A=\mathbb{Z}[x_1,\ldots ,x_k]$, the following triples work:
\begin{itemize} 
  \item $G=\mathrm{SL}_{m}(A)$ with $m\geq 3$; $H=\mathrm{SL}_2(A)$ in the left upper corner of $G$; and $U=\bigcup_{h\in H}h\mathcal{E}h^{-1}$.
  \item $G=\mathrm{Sp}_{2m}(A)$ with $m\geq 2$; $H=\mathrm{Sp}_2(A)$ realized in $G$ as in item $(iii)$ of Definition~$\ref{def:idensym}$; and $U=\bigcup_{h\in H}h\mathcal{S}h^{-1}$.
\end{itemize}
In these cases, it is easy to check $(*)$.
\end{rem}

\begin{rem}\label{rem:TTq}
As in Section~\ref{sec:prem}, it is not known whether universal lattices have $(\mathrm{TT})$. However by a matrix computation is \cite{PaWo} and a previous result of the author \cite{Mim2}, we have at least the following:

\begin{prop}\label{pro:univTT}
Let $A=\mathbb{Z}[x_1,\ldots ,x_k]$ with $k\in \mathbb{N}$. Suppose $m\geq 6$. Then $\mathrm{SL}_m(A) \geqslant \mathrm{U}_2(A)$ has relative property $(\mathrm{TT})$. Here $\mathrm{U}_2(A)$ denotes the multiplicative group of unipotent matrices in $\mathrm{M}_2(A)$, which sits in the left upper corner of $\mathrm{SL}_m(A)$.
\end{prop}

\begin{proof}
Thanks to $(\mathrm{TT})/\mathrm{T}$, we only need to verify boundedness of quasi-cocycles into trivial representation, which is called a \textit{quasi-homomorphisms} \cite{Cal}. It is easy to see that such a quasi-cocycle is bounded on the set of elementary matrices (because every elementary matrices in $\mathrm{SL}_m(A)$ is conjugate to its inverse). 
We use the fact that the ring $A$ is a unique factorization domain. Therefore any $2\times 2$ unipotent matrix $s$ is of the following form:
$$
s=\left(
\begin{array}{cc}
1+auv & -au^2 \\
av^2  & 1-auv
\end{array}
\right).
$$
Here $a,u,v\in A$. P. 281 of \cite{PaWo} (Mennicke) states this matrix is a product of $8$ elementary matrices in $\mathrm{E}_3(A)(=\mathrm{SL}_3(A))$. Finally, we appeal to the following result (recall that a quasi-homomorphism is said to be \textit{homogeneous} if it is a homomorphism on every cyclic subgroups):
\begin{thm}$($Compare with \cite[Claim, p. $3524$]{Mim2}$)$\label{thm:key}
Let $R$ be a $($possibly noncommutative$)$ associative unital ring. Suppose $g\in \mathrm{E}_2(R)$ and $s\in \mathrm{E}_2(R)$ satisfy the following conditions:
\begin{enumerate}[$(i)$]
  \item The matrix $g+g^{-1}\in \mathrm{M}_2(R)$ commutes with $s$.
  \item The equality $(s-I_2)^2=0$ holds.
\end{enumerate}
Then for any homogeneous quasi-homomorphism $\phi$ on $\mathrm{E}_6(R)$, we have
$$
\phi(gs)=\phi(sg)= \phi(g)+\phi(s).
$$
Here we put $\mathrm{E}_2(R)$ in the left upper corner of $\mathrm{E}_6(R)$, and view $g,s\in\mathrm{E}_2(R)$ as elements in $\mathrm{E}_6(R)$ with this identification.
\end{thm}
This theorem can be shown in the same argument as in \cite{Mim2}. Since for any $g\in \mathrm{E}_2(R)$, $g+g^{-1}$ is in the center of $\mathrm{M}_2(R)$ provided $R$ is commutative, Theorem~\ref{thm:key} shows that every quasi-homomorphism on $\mathrm{SL}_m(A)$ is bounded on $\mathrm{U}_2(A)$. This ends our proof (usually the terminology ``quasi-homomorphism" is used for the case of $(\pi,\mathfrak{H})=(1_G,\mathbb{R})$. However Theorem~\ref{thm:key} holds for any quasi-cocycle into trivial representation on any Banach space).
\end{proof}

We warn the following: in view of Theorem~\ref{thm:vassp}, it may seem that Proposition~\ref{pro:univTT} would imply $(\mathrm{TT})$ for universal lattices of degree $\geq 6$. However, there is a gap, which lies in the fact that $\mathrm{U}_2(\mathbb{Z}[x])$ is a quite \textit{small} subgroup in $\mathrm{SL}_2(\mathbb{Z}[x])$. See \cite{KrMc} for precise meaning of this.
\end{rem}

\section{Measure equivalence and induction of quasi-cocycles}\label{sec:ME}
\subsection{Definition and ME-cocycles}\label{subsec:ME}
We shortly recall definition of the measure equivalence and related concepts. For comprehensive treatments, we refer the readers to surveys of A. Furman \cite{Fur2} and of Shalom \cite{Sha4}.

\begin{defn}\label{def:me}
\begin{enumerate}[$(i)$]
 \item Two countable (infinite discrete) group $\Gamma,\Lambda$ are \textit{measure equivalent} (or shortly \textit{ME}) if there exists an infinite measure space $(\Omega,m)$ with a measurable, measure preserving action of $\Gamma\times \Lambda$ such that the action of each of the groups $\Gamma,\Lambda$ admit \textit{finite measure} fundamental domains:
 $$
 \Omega= \bigsqcup_{\gamma\in \Gamma}\gamma Y=\bigsqcup_{\lambda\in \Lambda}\lambda X.
 $$
 The space $(\Omega,m)$ is called a $(\Gamma,\Lambda)$\textit{-coupling} or an \textit{ME-coupling}.
 \item Suppose $(\Omega,m)$ be an ME-coupling of two groups $\Gamma$ and $\Lambda$ and let $Y,X$ be fundamental domains for $\Gamma$, $\Lambda$ actions respectively. The \textit{retraction} associated with $X$ is the measurable $\Lambda$-equivariant map $\kappa\colon \Omega \to \Lambda$ defined as follows: for $\omega\in \Omega$, $\kappa(\omega)^{-1}\omega \in X$. From this, we obtain a map 
 $\alpha\colon \Gamma \times X \to \Lambda$ defined by 
 $ \alpha(\gamma ,x)=\kappa (\gamma x)^{-1}$ for $\gamma\in \Gamma$, $x\in X$. Similarly, one can define the retraction associated with $Y$ and a map $\beta \colon \Lambda \times Y \to \Gamma$. These maps $\alpha,\beta$ are called \textit{ME-cocycles}. 
 \end{enumerate}
 \end{defn}
A motivating example is that $\Gamma$ and $\Lambda$ are lattices in the same locally compact group $G$. Then $G$ with a Haar measure gives an ME-coupling (recall that existence of a lattice implies $G$ is unimodular), and a corresponding $\Gamma\times \Lambda$ action is the left multiplication of $\Gamma$ and the right inverse multiplication of $\Lambda$ on $G$.

We have a natural $\Gamma$-action on $X$, and denote by $\gamma \cdot x$ ($\gamma\in \Gamma, x\in X$) to distinguish it from the $\Gamma$-action on $\Omega$. In this setting, the ME-cocycle $\alpha\colon \Gamma \times X \to \Lambda$ satisfies the following:
$$
\textrm{for any $\gamma\in \Gamma$ and $x\in X$, }\gamma \cdot x =\alpha(\gamma, x) \gamma x.
$$
Therefore this $\alpha$ is a \textit{Borel cocycle}, namely, a Borel measurable map satisfying the \textit{cocycle identity}:
$$
\textrm{for any $\gamma_1,\gamma_2\in \Gamma$ and $x\in X$, \ \ }
\alpha (\gamma_2 \gamma_1,x)=\alpha (\gamma_2,\gamma_1 \cdot x)\alpha(\gamma_1,x).
$$

Note that ME is indeed an equivalence relation, see \cite{Fur2} for the proof. Also, for any ME coupling, one can replace it with \textit{ergodic} one, see Lemma 2.2 of \cite{Fur1}.

\subsection{Induction process}\label{subsec:ind}
One good feature of ME is that one can induce unitary representation and bounded cohomology. Also, induction of  (quasi-)cocycles is available if some integrability condition (recall the statement in Theorem~\ref{mthm:ME}) is fulfilled. Here we only treat $L^2$-induction of a Hilbert (unitary) $\Gamma$-module, which fits our purpose. For our use in below, we will consider $\Gamma$-induction of $\Lambda$-module, for an ME-coupling $(\Omega,m)$ of $\Gamma$, $\Lambda$.

Let $(\sigma, \mathfrak{H})$ be a unitary $\Lambda$-representation. First, we define the induced $\Gamma$-module $(\Omega \mathbf{I}_{\Lambda}^{\Gamma}\sigma, \Omega_{\Lambda}^{\Gamma}\mathfrak{H})$ by
\begin{align*}
\Omega_{\Lambda}^{\Gamma}\mathfrak{H}&=L^{[2]}(\Omega ,\mathfrak{H})^{\Lambda} \\
&=\{ f\colon \Omega \to \mathfrak{H} \textrm{ with }f(\lambda \omega)=\sigma(\lambda) f(\omega)\textrm{ and }\|f\|^2_2= \int_{X} \|f(x)\|^2dm(x)<\infty \},
\end{align*}
where $X\cong \Omega/\Lambda$ is a $\Lambda$-fundamental domain and the $\Gamma$-representation is given by translation. Equivalently, $\Omega_{\Lambda}^{\Gamma}\mathfrak{H}\cong L^2(X,\mathfrak{H})$ with the twisted $\Gamma$-action defined a.e. by $$
((\Omega \mathbf{I}_{\Lambda}^{\Gamma}\sigma)(\gamma) f)(x)=
\sigma (\alpha(\gamma^{-1},x)^{-1})f(\gamma^{-1}\cdot x)
$$
for the associated ME-cocycle $\alpha\colon \Gamma \times X \to\Lambda$. Then $\Omega \mathbf{I}_{\Lambda}^{\Gamma}\sigma$ becomes a unitary $\Gamma$-representation. 

Secondly, we see induction of the bounded cohomology. Fix a $\Lambda$-fundamental domain $X$ and a retract $\kappa\colon \Omega \to \Lambda$. For $f\in \ell^2(\Lambda^{n+1};\mathfrak{H})^{\Lambda}$, define 
$$
\Omega \mathbf{i}_{\Lambda}^{\Gamma}f(\gamma_0,\ldots ,\gamma_n)(\omega):=
f(\kappa(\gamma_0^{-1}\omega),\ldots ,\kappa(\gamma_n^{-1}\omega)) \in \ell^{\infty}(\Gamma^{n+1};L^{[2]}(\Omega,\mathfrak{H})^{\Lambda})^{\Gamma}.
$$
The point here is that $\Omega \mathbf{i}_{\Lambda}^{\Gamma}f$ is $2$-integrable because we treat bounded cocycles (this is why the induction process for \textit{bounded} cohomology behaves goodly).

This map induces the well-defined morphism
$$
\Omega \mathbf{i}_{\Lambda}^{\Gamma}\colon H^n_{\mathrm{b}}(\Lambda;\sigma) \to H^n_{\mathrm{b}}(\Gamma;\Omega \mathbf{I}_{\Lambda}^{\Gamma}\sigma).
$$

Finally, we explain induction of (quasi-)cocycles. Let $b\colon \Lambda \to \mathfrak{H}$ be a quasi-$\sigma$-cocycle. Then by the identification $\Omega_{\Lambda}^{\Gamma}\mathfrak{H}\cong L^2(X,\mathfrak{H})$, we hope to define the induced $\Omega \mathbf{I}_{\Lambda}^{\Gamma}\sigma$-cocyle $\tilde{b}$ by the following formula:
$$
\textrm{for $\gamma\in \Gamma$ and $x\in X$, \ }\tilde{b}(\gamma)(x):=b(\alpha(\gamma^{-1},x)^{-1}).
$$
However, this is \textit{not} necessarily $2$-summable (, namely, possibly $\tilde{b}$ does \textit{not} range in $L^2(X,\mathfrak{H})$). One condition, called the $L^2$-condition, assures that the induction above is well-defined:
\begin{lem}\label{lem:2ind}
Stick to the setting in this subsection. Suppose $\Lambda$ is finitely generated and there exists a $\Lambda$-fundamental domain $X$ such that the associated ME cocycle $\alpha \colon \Gamma \times X \to \Lambda$ satisfies the $L^2$-$\mathrm{condition}$: 
$$
\mathrm{for\ any\ \gamma \in \Gamma ,\ }|\alpha(\gamma,\cdot )|_{\Lambda} \in L^2(X).
$$
Here $|\cdot|_{\Lambda}$ is a word metric on $\Lambda$ with respect to a finite generating set of $\Lambda$. Then for any unitary $\Lambda$-representation $(\sigma, \mathfrak{H})$ and any quasi-$\sigma$-cocycle $b$, the $(L^2$-$)$induction $\tilde{b}$ becomes a quasi-$\Omega\mathbf{I}_{\Lambda}^{\Gamma}\sigma$-cocycle $\Gamma \to L^2(X,\mathfrak{H})$.

In particular, the induction map of bounded cohomology
$$
\Omega \mathbf{i}_{\Lambda}^{\Gamma}\colon H^n_{\mathrm{b}}(\Lambda;\sigma) \to H^n_{\mathrm{b}}(\Gamma;\Omega \mathbf{I}_{\Lambda}^{\Gamma}\sigma),
$$
associated with this fundamental domain $X$ maps $\widetilde{QH}(\Lambda ;\sigma)$ inside $\widetilde{QH}(\Gamma ;\Omega\mathbf{I}_{\Lambda}^{\Gamma}\sigma)$.
\end{lem}
This lemma follows from the fact that the growth of a quasi-cocycle can be dominated linearly by a word length of a group (note that if $\Lambda$ is finitely generated, the choice of finite generating sets does not affect the growth of the associated word length), and from Lemma~\ref{lem:comp}.

\begin{rem}\label{rem:2-sum}
Some cases this $L^2$-condition is known to be held. A trivial case is that $\Gamma$ and $\Lambda$ are lattices in the same group, and that $\Gamma$ is cocompact. In this case, an ME-cocycle $\alpha$ (for an appropriate $\Lambda$-fundamental domain) satisfies $L^{\infty}$-condition, which is stronger than $L^2$-condition. 

A quite non-trivial case is that $\Lambda$ (or $\Gamma$) is a lattice in a simple higher rank algebraic group. In this case, an argument in Shalom's paper \cite{Sha2}, which is based on a study of word lengths on higher rank lattice by Lubotzky--Mozes--Raghunathan \cite{LMR}, together with Furman's superrigidity \cite{Fur1} insures that there exists a $\Lambda$-fundamental domain such that the associated ME-cocycle satisfies $L^p$-condition for all finite $p$.
\end{rem}

For the proof of Theorem~\ref{mthm:ME}, certain properties concerning induction are needed. First, we need the conception of \textit{weak mixing property} for a unitary representation, see \cite{Fur1} for details. We just recall the definition that a unitary representation is said to be \textit{weakly mixing} if it does not contain finite dimensional subrepresentations. An example of weakly mixing representation is the left regular representation of an infinite countable group (because matrix coefficients are $C_0$ in this case, in other words, it is \textit{strongly mixing}). We also note that weak mixing property is stable under pull-backs and inductions from finite index subgroups.

\begin{lem}$($Furman \cite[Lemma $8.2$]{Fur1}$)$\label{lem:fur}
Let $(\Omega,m)$ be an ergodic ME-coupling of $\Gamma$ and $\Lambda$ and let $\sigma$ be a weakly mixing unitary $\Lambda$-representation. Then the induced $\Gamma$-representation $\Omega\mathbf{I}_{\Lambda}^{\Gamma}\sigma$ does not contain $1_{\Gamma}$.
\end{lem}

Finally, we state a notable result of Monod--Shalom which concerns induction of bounded cohomology (they in fact prove for any separable coefficient module):
\begin{thm}$($Monod--Shalom \cite[Theorem $4.4$]{MoSh}$)$\label{thm:MS}
Let $(\Omega, m)$ be an $ME$-coupling of $\Gamma,\Lambda$. Then the induction map in degree $2$
$$
\Omega \mathbf{i}_{\Lambda}^{\Gamma}\colon H^2_{\mathrm{b}}(\Lambda;\sigma) \to H^2_{\mathrm{b}}(\Gamma;\Omega \mathbf{I}_{\Lambda}^{\Gamma}\sigma),
$$
is injective for any Hilbert $($unitary$)$ $\Lambda$-module. Moreover, it does not depend on the choice of a retraction $\kappa$.
\end{thm}

\subsection{Proof of Theorem~\ref{mthm:ME}}
\begin{proof}
First, we observe that in the proof Theorem~\ref{thm:TTMCG}, the following two properties are sufficient to deduce homomorphism superrigidity from a countable group $\Lambda$ into $\mathrm{MCG}(\Sigma)$ and into $\mathrm{Out}(F_n)$:
\begin{enumerate}[$(a)$]
  \item  For any \textit{weakly mixing} unitary $\Lambda$ representation $\sigma$, $\widetilde{QH}(\Lambda,\sigma)=0$;
  \item  For any finite index subgroup $\Lambda_0\leq \Lambda$, $\Lambda_0$ has finite abelianization (for instance, $\Lambda$ has property $(\mathrm{T})$).
\end{enumerate}
For weakly mixing representation, see the previous subsection. 

In the setting of Theorem~\ref{mthm:ME}, we will check that $\Lambda$ satisfies these two properties. Item $(b)$ follows from the fact that $(\mathrm{T})$ is an ME-invariant \cite{Fur2}. For item $(a)$, take any weakly mixing $\Lambda$-representation $\sigma$. Consider the $L^2$-induction of this representation and corresponding bounded cohomology. Theorem~\ref{thm:MS}, Lemma~\ref{lem:2ind}, and existence of a $\Lambda$-fundamental domain with respect to which the ME-cocycle is $L^2$ imply that the induction map in degree $2$
$$
\Omega \mathbf{i}_{\Lambda}^{\Gamma}\colon H^2_{\mathrm{b}}(\Lambda;\sigma) \to H^2_{\mathrm{b}}(\Gamma;\Omega \mathbf{I}_{\Lambda}^{\Gamma}\sigma)
$$
is injective and that it maps $\widetilde{QH}(\Lambda;\sigma)$ into $\widetilde{QH}(\Gamma;\Omega \mathbf{I}_{\Lambda}^{\Gamma}\sigma)$ as an injection. By Lemma~\ref{lem:fur}, $\Omega \mathbf{I}_{\Lambda}^{\Gamma}\sigma \not\supseteq 1_{\Gamma}$. Hence thanks to property $(\mathrm{TT})/\mathrm{T}$ for $\Gamma$ (Theorem~\ref{thm:TTmodT}), $\widetilde{QH}(\Gamma;\Omega \mathbf{I}_{\Lambda}^{\Gamma}\sigma)=0$. We thus have
$$
\widetilde{QH}(\Lambda;\sigma)=0
$$
and this confirms $(a)$ in above.
\end{proof}

\begin{rem}\label{rem:sar}
Furman has pointed out to the author the following: his work in progress with U. Bader and R. Sauer \cite{BFS} implies that at least for a universal lattice or a symplectic universal lattice itself, if a locally compact second countable group $G$ contains  it as a lattice, then up to finite noise and compact kernel, $G$ must be totally disconnected. It might be interesting if there exists a nontrivial $\Lambda$ in Theorem~\ref{mthm:ME} (for $A$ which is not a ring of integers of a certain field). However at the moment, the author has no such examples.
\end{rem}

\begin{rem}
It does not seem to be known whether the $L^2$-condition in Theorem~\ref{mthm:ME} can be removed. One possible way to do this is to show that $H^2_{\mathrm{b}}(G;\pi)$ itself vanishes for $G$ a universal lattice or symplectic universal lattices; and $\pi \not\supseteq 1_G$ a unitary representation. Compare with Conjecture~1.8 in \cite{Mim1}.

In fact, by making full use of infinite dimensionality in Theorem~\ref{thm:ham} and Theorem~\ref{thm:BBF}, we only have to show finite dimensionality of $H^2_{\mathrm{b}}(G;\pi)$ above; or $H^2(G;\pi)$ in the setting above. The latter follows from the injectivity of the comparison maps, which is a corollary of $(\mathrm{TT})/\mathrm{T}$.

\end{rem}

\section{Property $(\mathrm{FF}_B)/\mathrm{T}$ for Banach spaces}\label{sec:tba}\subsection{Definitions and fundamental facts}\label{subsec:tba}
A Banach space $B$ is said to be \textit{ucus} if the norm is uniformly convex and uniformly smooth (for these definitions details, we refer to \cite{BFGM}), and is said to be \textit{superreflexive} if it admits a ucus norm which is compatible to the original norm. Superreflexive Baanach spaces are reflexive. A basic example of superreflexive space is $L^p$-space with $p\in (1,\infty)$ (it is ucus). 

In this section, we consider an isometric representation $\rho$ of a group $G$ in a Banach space $B$. Unlike unitary cases, it is not true in general that the space $B^{\rho(G)}$ of $\rho(G)$-invariant vectors is complemented in $B$. However, the following result in \cite{BFGM} states that it is complemented if $B$ is superreflexive:
\begin{prop}$($\cite[Proposition $2.6$, Proposition $2.10$]{BFGM}$)$\label{pro:decom}
Let $G$ be a group, $N\trianglelefteq G$ be a $\mathrm{normal}$ subgroup. Let $\rho$ be an isometric $G$-representation on a ucus Banach space $B$. Then there is a decomposition of $B$ as $\rho(G)$-spaces
$$
B=B^{\rho(N)} \oplus B'_{\rho(N)}.
$$
Here $B'_{\rho(N)}$ as the annihilator of $(B^{*})^{\rho^{\dagger}(N)}$, where $\rho^{\dagger}$ denotes the $\mathrm{contragredient}$ $\mathrm{representation}$  $G$ in $B^*$, defined as $\rho^{\dagger}(g)\phi \rangle = \langle \rho (g^{-1})\xi, \phi \rangle$ $(g\in G, \phi \in B^*, \xi \in B)$ ($\langle \cdot ,\cdot \rangle$ denotes the duality $B\times B^* \to \mathbb{C}$). 

Moreover,  for any $\xi=\xi_0 +\xi_1$ where $\xi\in B$, $\xi_0\in B^{\rho(N)} $ and $\xi_1\in B'_{\rho(N)}$, $\| \xi_0 \| \leq \| \xi \|$ and $ \| \xi_1 \| \leq 2\| \xi \|$ hold. 
\end{prop}
We note that in \cite{BFGM} (Proposition~2.3), they also show that for a superreflexive space $B$ and an isometric $G$-representation $\rho$ in $B$, one can take a $\rho(G)$-invariant \textit{ucus} norm which is compatible to the original norm. Also note that the normality of the subgroup $N$ is needed, otherwise $B^{\rho(N)}$ may not be $\rho(G)$-invariant.

For an isometric $G$-representation $\rho$ in $B$, the concepts of $\rho$ having \textit{almost invariant vectors} (we also write this as $\rho \succeq 1_G$); $\rho$-(1-)\textit{cocycles}, $\rho$-(1-)\textit{coboundaries}, and \textit{group cohomology} $H^{\bullet}(G;\rho)$; \textit{quasi-}$\rho$-(1-)\textit{cocycles}; and bounded cohomology $H_{\mathrm{b}}^{\bullet}(G;\rho)$ are defined in the same manner as in Definition~\ref{def:T}.

Recall that property $(\mathrm{T})$ for a group is defined by the condition ``for any unitary $G$-representation, $\pi \nsupseteq 1_G  \Rightarrow \pi \nsucceq 1_G$." However, in general setting, the straight generalization of above does not give information for the case of $\pi \supseteq 1_G$. More precisely, if we consider unitary representation, then we can restrict our representation on the orthogonal complement of the space of invariant vectors. However even in the case of considering isometric representation on $L^p$ spaces, the canonical complement, defined in Proposition~\ref{pro:decom} is a subspace of $L^p$ space, and usually \textit{not} realizable as an $L^p$ space on any measure space. Therefore, the following definition is appropriate:

\begin{defn}$($\cite{BFGM}$)$\label{def:propertyTB}
Let $B$ be a Banach space. 
\begin{enumerate}[$(1)$]
  \item A pair  $G \trianglerighteq N$ of a group and a \textit{normal} subgroup is 
  said to have \textit{relative property} $(\mathrm{T}_B)$  if for any 
  isometric representation $\rho$ of $G$ in $B$, the 
  isometric representation $\rho '$ on the quotient Banach space  $B / B^{\rho (N)}$, naturally induced by $\rho$, satisfies 
  $  \rho' \nsucceq 1_G$.   
  By Proposition~\ref{pro:decom}, if $B$ is superreflexive, then the definition above is equivalent to the following: for any 
  isometric representation $\rho$ of $G$ in $B$, 
  the restriction of $\rho$ in $B_{\rho(N)}'$ (see Proposition~\ref{pro:decom}) does not have almost invariant vectors.

  A group $G$ is said to have \textit{property} 
  $(\mathrm{T}_B)$ if $G \trianglerighteq G$ has 
  relative $(\mathrm{T}_B)$.
  \item  A group  $G$ is said to have \textit{property }$(\mathrm{F}_B)$ if for any isometric $G$-representation $\rho$ in $B$, every $\rho$-cocycle is a $\rho$-coboundary. Equivalently, if for any such $\rho$, $H^1(G;\rho,B)=0$ holds.
  \item (\cite{Mim1}) A pair  $G\supseteq U$ is said to have \textit{relative property }$(\mathrm{FF}_B)$ if for any isometric $G$-representation $\rho$ in $B$, every quasi-$\rho$-cocycle is bounded on $U$. A group  $G$ is said to have \textit{property }$(\mathrm{F}_B)$ if the pair $G\supseteq G$ has relative $(\mathrm{FF}_B)$.
  \item (\cite{Mim1}) A group $G$ is said to have \textit{property} 
   $(\mathrm{FF}_B)/\mathrm{T}$  if for any isometric $G$-representation $\rho$  in $B$  and    any quasi-$\rho$-cocycle $b$, $b'(G)$ is 
   bounded, where $b'\colon \Gamma \to B/B^{\rho(G)}$ is the natural quasi-cocycle constructed from the projection of $b$ associated the canonical quotient map $B\twoheadrightarrow B/B^{\rho(G)}$. If $B$ is superreflexive, then this definition is equivalent to the following condition: for any isometric representation $\rho$ of $\Gamma$ in $B$ and 
   any quasi-$\rho$-cocycle $b$, $b_1(G)$ is 
   bounded. Here we decompose $b$ as $b_0+b_1$ such that $b_0$ takes 
   values in $B_0=B^{\rho (\Gamma)}$ and $b_1$ takes values in 
   $B_1=B'_{\rho(\Gamma)}$. 
\end{enumerate}
If $B=\mathcal{B}$ is a class of Banach spaces, then we define these properties above as having corresponding property for all Banach spaces in the class $\mathcal{B}$.
\end{defn}
The well-known lemma of Chebyshev center states that for a bounded subset $X$ of a uniformly convex Banach space $B$, there exists a unique closed ball with minimimum radius which contains $X$ (the center of this ball is called the Chabyshev center). This implies the following: for a ucus Banach space $B$, every \textit{bounded} cocycle into an isometric representation in $B$ is a coboundary. Therefore for a superreflexive Banach space, $(\mathrm{FF}_B)$ implies $(\mathrm{F}_B)$, and $(\mathrm{FF}_B)$ for a group $G$ is equivalent to the following two conditions:
\begin{itemize}
  \item for any isometric $G$-representation $\rho$, $H^1(G;\rho)=0$;
  \item for any isometric $G$-representation $\rho$, $H^2_{\mathrm{b}}(G;\rho)$ naturally injects into $H^2(G;\rho)$.
\end{itemize}

We will employ the following two lemmas in the next subsection:
\begin{lem}\label{lem:cheb}
Suppose $B$ is uniformly convex. For an isometric representation $\rho$ of a group $G$ in $B$, if there exists $\xi\in B$ such that $\sup_{g\in G}\| \xi -\rho(g)\xi\| <\|\xi\|$, 
then $\rho \supseteq 1_G$.
\end{lem}
This follows from the lemma of a Chebyshev center.
\begin{lem}$($\cite[Lemma~$2.10$]{Mim1}$)$\label{lem:kazhconstTB}
Suppose $B$ is us, $G$ is a finitely generated group and 
$S$ is a finite generating set of $G$. Let $N\trianglelefteq G$. 
Let  $\rho$  be any isometric representation of $G$ in 
$B$,  $\xi$ be any vector in $B$ and set $\delta_{\xi} := 
\sup_{s \in S} \|  \xi -\rho (s)\xi \|$. 
If a pair $G \trianglerighteq N$ has relative 
$(\mathrm{T}_B)$, then 
$$
\textrm{for any $l\in N$,\ \ }
\| \xi - \rho(l)\xi  \| \leq 4 \mathcal{K}^{-1} \delta_{\xi}. 
$$
Here $\mathcal{K}$ stands for the $\mathrm{relative}$ $\mathrm{Kazhdan}$ $\mathrm{constant}$ $\mathcal{K}(G,N;S,\rho) $ for $(\mathrm{T}_B)$, defined as 
$$\mathcal{K}(G,N;S,\rho):= \inf_{\xi \in S(B_1) }
      \sup_{s \in S}   \|  \xi -\rho (s)\xi  \| . 
$$ 
Here $B_1 =B_{\rho (N)}'$ and $S(B_1)$ means the unit sphere of $B_1$. 
\end{lem}
Note that if $G$ is finitely generated, then relative $(\mathrm{T}_B)$ for $G\trianglerighteq N$ exactly says that $\mathcal{K}(G,N;S,\rho)>0$ for any finite generating set $S\subseteq G$ and any isometric $G$-representation $\rho$ in $B$.

Notable studies have been done for the case of $B=\mathcal{L}_p$, the class of $L^p$ spaces, with $p\in (1,\infty)$ by P. Pansu \cite{Pan}; Bourdon--Pajot; G. Yu \cite{Yu}; and Bader--Furman--Gelander--Monod \cite{BFGM}. We shortly states results here. Let for each $p$, $\mathcal{L}_p$ denotes the class of all $L^p$-spaces.
\begin{thm}\label{thm:BFGM1}
Let $G$ be a locally compact and second countable group. 
\begin{enumerate}[$\mathrm{(}$i$\mathrm{)}$]
   \item $($\cite{BFGM}$)$ For any  Banach space $B$, 
        property $(\mathrm{F}_B)$ implies property $(\mathrm{T}_B)$.
   \item $($\cite{BFGM}$)$ Property $(\mathrm{T})$ is equivalent to property 
   $(\mathrm{T_{\mathcal{L}_p}})$, where $p \in (1, \infty)$. 
   It is also equivalent to 
    property $(\mathrm{F_{\mathcal{L}_p}})$, 
    where $p \in (1, 2]$. 
    
    Moreover, relative property $(\mathrm{T})$ $($for a pair of a group and a normal subgroup$)$ is equivalent to relative property $(\mathrm{T}_{\mathcal{L}_p})$ for every $p\in (1,\infty)$.
    \item $($\cite{BFGM}$)$ Any totally higher rank lattices, in the sense in Subsection~$\ref{subsec:TT}$,  $\Gamma$ have property  $(\mathrm{F_{\mathcal{L}_p}})$. 
    \item $($\cite{Yu}$;$\cite{Pan}$)$ Any hyperbolic group $H$, including one with $(\mathrm{T})$, has $p_0\in (1,\infty)$ such that for any $(\infty >)p>p_0$, $H$ is $\mathcal{L}_p$-$\mathrm{Haagerup}$, namely, $H$ admits a metrically proper cocycle into an isomtric representation in an $L^p$-space. In particular, $\mathrm{Sp}_{n,1}$ $\mathrm{fails}$ to have $(\mathrm{F}_{\mathcal{L}_p})$ for $p>4n+2$.
\end{enumerate}
\end{thm}
These result imply the following two points: first, $(\mathrm{F}_{\mathcal{L}_p})$ is strictly stronger than $(\mathrm{T})$$\Leftrightarrow(\mathrm{FH})=(\mathrm{F_{\mathcal{L}_2}})$ if $p\gg 2$; and secondly, thus $(\mathrm{F}_{B})$ is stronger than $(\mathrm{T}_{B})$ in general.

\subsection{Proof of Theorem~\ref{thm:FFLp}}\label{subsec:prooftba}
First, in exactly the same argument as one in Theorem~\ref{thm:short} we have the following result:

\begin{thm}\label{thm:shortFB}
Let $B$ be a Banach space or a class of them. Suppose a triple $(G,H,U)$, where $G$ is a countable group; $H\leqslant G$ is a subgroup; and $U\subseteq G$ is a closed subset, satisfies both conditions $(i)$, $(ii)$, $(iii)$ in Theorem~$\ref{thm:short}$ and the following two conditions:
   \begin{itemize}
     \item[$(iv')$] $G\supseteq U$ has relative $(\mathrm{FF}_B)$;
     \item[$(v')$] $G$ has $(\mathrm{T}_B)$.
    \end{itemize}
Then $G$ has property $(\mathrm{FF}_B)/\mathrm{T}$.
\end{thm}
Note that Theorem~\ref{thm:shortFB} deduces $(\mathrm{FF}_B)/\mathrm{T}$ from $(\mathrm{T}_B)$ under some condition. Since $(\mathrm{T}_B)$ is much weaker than $(\mathrm{F}_B)$, this implication might be powerful.

The following proposition is a key to establish $(\mathrm{FF}_{\mathrm{L}_p})$ for universal lattices of degree $\geq 4$ in the paper \cite{Mim1} of the author:
\begin{prop}$($\cite[Theorem~$6.4$]{Mim1}$)$\label{pro:TtoFF}
Let $A = \mathbb{Z} [x_1, \ldots , x_k]$. 
Suppose $B$ is any superreflexive Banach space. Then, if the pair 
$ \mathrm{E}_2 (A) \ltimes A^2 \trianglerighteq A^2$ has relative property 
$(\mathrm{T}_B)$, then the pair $ \mathrm{SL}_3 (A ) \ltimes A^3 \geqslant A^3$ 
has relative property $(\mathrm{FF}_B)$.
\end{prop}

To prove Theorem~\ref{thm:FFLp}, we will show the following symplectic version of the proposition above: 
\begin{prop}\label{prop:TtoFFsp}
Let $A = \mathbb{Z} [x_1, \ldots , x_k]$. 
Suppose $B$ is any superreflexive Banach space. If the pairs 
$ \mathrm{E}_2 (A) \ltimes A^2 \trianglerighteq A^2$; and 
$ \mathrm{E}_2 (A) \ltimes S^{2*}(A^2) \trianglerighteq S^{2*}(A^2)$ have relative property 
$(\mathrm{T}_B)$, then the pair $ \mathrm{SL}_3 (A ) \ltimes S^{3*}(A^3) \geqslant S^{3*}(A^3)$ 
has relative property $(\mathrm{FF}_B)$.
\end{prop}

\begin{proof}
Set $G=\mathrm{SL}_3 (A ) \ltimes S^{3*}(A^3)$ and $N=S^{3*}(A^3) \trianglelefteq G$, and consider $G$ (and $N$) as subgroup(s) in $\mathrm{Sp}_6(A)$, as in item $(ii)$ of Definition~\ref{def:idensym}. Let $\rho$ be an isometric $G$-representation in $B$ and $b\colon G\to B$ be an (arbitrary) quasi-$\rho$-cocycle. Take a ucus $\rho$-invariant norm on $B$ and fix it (see the paragraph below Proposition~\ref{pro:decom}). Take a decomposition $B=B_0\oplus B_1:=B^{\rho(N)}\oplus B_{\rho(N)}'$, and decompose $b$ as $
b=b_0+b_1$, 
where $b_0\colon G\to B_0$ and $b_1\colon G\to B_1$. Then for $i=1,2$, by $\rho(G)$-invariance of $B_i$, $b_i$ becomes a quasi-$\rho$-cocycle. Set the following two subgroups $N_1$, $N_2$ of $N$:
\begin{align*}
N_1&:=\left\{ B_{1,1}(r)=\left(
   \begin{array}{c|ccc}
     {}  & r & 0 & 0 \\
      I_3 & 0 & 0 & 0 \\
      {} & 0 & 0& 0 \\
     \hline  
     0 & {}& I_3 & {}
   \end{array}\right)
   : r\in A\right\}, \\
N_2&:=\left\{ B_{1,2}(r)=\left(
   \begin{array}{c|ccc}
     {}  & 0 & r & 0 \\
      I_3 & r & 0 & 0 \\
      {} & 0 & 0& 0 \\
     \hline  
     0 & {}& I_3 & {}
   \end{array}\right)
   : r\in A\right\}. 
\end{align*}

Then in a similar way to one in the proof of \cite{Mim1}[Theorem~1.3], the following two can be verified:
\begin{enumerate}[$(1)$]
  \item The quasi-cocycle $b_0$ is bounded on $N$.
  \item If $b_1$ is bounded both on $N_1$ and on $N_2$, then it is bounded on $N$.
\end{enumerate}
Therefore for the proof of the theorem, it suffices to show the following two assertions:
\begin{itemize}
  \item[(A$1)$] \textit{The set $b_1(N_1)$ is bounded.}
  \item[(A$2)$] \textit{The set $b_1(N_2)$ is bounded.}
\end{itemize}  

In below, we shall prove assertions $(\mathrm{A}1)$ and $(\mathrm{A}2)$. Recall from Definition~\ref{def:symp} the definitions of $B_{i,j}(r)$ and $D_{i,j}(r)$. Also recall from Lemma~\ref{lem:spfg} a finite generating set of $\mathrm{Sp}_6(A)$. Set
\begin{align*}
S&=\{ B_{i,j}(\pm x_l), D_{i,j}(\pm x_l): 1\leq i,j\leq 3, i\ne j, 0\leq l\leq k\} \\
&\cup \{B_{i,i}(\pm x_1^{\epsilon_1}\cdots x_k^{\epsilon_k}):1\leq i\leq 3, \epsilon_1 ,\ldots ,\epsilon_k\in \{0,1\} \},
\end{align*}
where $x_0=1$ and $x_l^0=1$. Then $S$ is a finite generating set of $G(\leqslant \mathrm{Sp}_6(A))$. 

Firstly, we verify assertion $(\mathrm{A}1)$. Define a \textit{finite} subset $S_0$ of $G$ as follows:
\begin{align*}
S_1&:=\{ B_{1,2}(\pm x_l), B_{1,3}(\pm x_l), B_{2,3}(\pm x_l), D_{1,2}(\pm x_l), D_{1,3}(\pm x_l), D_{2,3}(\pm x_l), D_{3,2}(\pm x_l):  0\leq l\leq k\} \\
&\ \ \cup \{B_{i,i}(\pm x_1^{\epsilon_1}\cdots x_k^{\epsilon_k}):1\leq i\leq 3, \epsilon_1 ,\ldots ,\epsilon_k\in \{0,1\} \} \\
&= \left\{ \left(
\begin{array}{ccc|ccc} 
 1 & * & * & * & * & * \\
 0 & 1 & * & * & * & * \\
 0 & * & 1 & * & * & * \\
 \hline 
 {} & {} & {} & 1 & 0 & 0 \\
 {} & 0 &  {} & * & 1 & * \\
 {} & {} & {} & * & * & 1 
 \end{array}\right) \right\} \cap S.
\end{align*}
Set 
$$
C:=\max \{ \sup_{s\in S_1}\| b_1(s)\|, \sup_{g,h\in G}\| b_1(gh)-b_1(g)-\rho(g)b_1(h)\| \} <\infty.
$$
It is a key observation to this proof that $N_1$ commutes with $S_1$. Therefore for any $l\in N_1$ and $s\in S_1$, the following inequalities hold:
\begin{align*}
  &\| b_1(l)-\rho(s)b_1(l)\| \leq \| b_1(l) - b_1(sl)\| +\|b_1(s)\| +C \\
                            \leq & \| b_1(l) - b_1(sl)\| +2C = \| b_1(l) - b_1(ls)\| +2C \\
                            \leq & \| b_1(l)-b_1(l) -\rho(l)b_1(s)\| +3C \leq 4C.
\end{align*}
Note that the term in the very below is independent of the choice of $l\in N_1$. Now we define the following pairs of subgroups $(G',N')$; $(H,H')$ of $G$:
\begin{align*}
 G':&=\left\{ \left(
 \begin{array}{cc|cc}
  1 & 0 & 0 & 0 \\
  0 & W' & 0 & v' \\
  \hline 
  0 & 0 & 1 & 0 \\
  0 & 0 & 0 & {}^t(W')^{-1}
  \end{array}\right) : W'\in \mathrm{E}_2(A), v' \in S^{2*}(A^2)\right\} \\
  &\trianglerighteq \{ g\in G': W'=I_2 \}:=N'; \\
   H&:=\left\{ \left(
 \begin{array}{cc|cc}
  1 & {}^tu & 0 & 0 \\
  0 & W' & 0 & 0 \\
  \hline 
  0 & 0 & 1 & 0 \\
  0 & 0 & {}^t(W')^{-1}u & {}^t(W')^{-1}
  \end{array}\right) : W'\in \mathrm{E}_2(A), u \in A^2\right\} \\
  &\trianglerighteq \{ h\in H: W'=I_2 \}:=H'. 
\end{align*}
Set a real number $\mathcal{K}$ by the following formula:
$$
\mathcal{K}:=\min \{ \mathcal{K}(G',N';S\cap G',\rho\mid_{G'}), \mathcal{K}(H,H';S\cap H,\rho\mid_{H})\}
$$
(recall the definition of the relative Kazhdan constant for property $(\mathrm{T}_B)$ from Lemma~\ref{lem:kazhconstTB}). Then by the assumptions of relative property $(\mathrm{T}_B)$ for $\mathrm{E}_2 (A) \ltimes S^{2*}(A^2) \trianglerighteq S^{2*}(A^2)$; and $\mathrm{E}_2 (A) \ltimes A^2 \trianglerighteq A^2$, $\mathcal{K}$ is strictly positive. Therefore, we have the following inequalities by  the inequalities above and Lemma~\ref{lem:kazhconstTB}:
\begin{align*}
\textrm{for any $\xi\in b_1(N_1)$, } \ \ \|\xi -\rho(n')\| \leq 16\mathcal{K}^{-1}C, \textrm{and}\ \ &\|\xi -\rho(h')\| \leq 16\mathcal{K}^{-1}C.
\end{align*}
Here $n' \in N'$ and $h'\in H'$ are arbitrary elements. Note the following:
\begin{align*}
 N'=\{ B_{2,3}(r),B_{2,2}(r),B_{3,3}(r):r\in A\}, \ 
 H'=\{ D_{1,2}(r), D_{1,3}(r): r\in A\}.
\end{align*}

Now we observe the following correspondences among elementary symplectic matrices according to choices of alternating matrices ($J_3 \leftrightarrow L_3$):
$$
\begin{array}{cccc}
\textrm{alternating matrix$\colon$} &J_3& \longleftrightarrow &L_3 \\
\hline \hline
{} & B_{2,3}(r)& \longleftrightarrow &SE_{3,6}(r), \\
{} & D_{1,2}(r)& \longleftrightarrow &SE_{1,3}(r), \\
{} & D_{1,3}(r)& \longleftrightarrow &SE_{1,5}(r), \\
{}  & B_{2,2}(r)& \longleftrightarrow &SE_{3,4}(r), \\
{} & B_{3,3}(r)& \longleftrightarrow &SE_{5,6}(r), \\
\hline
{} & B_{1,2}(r)& \longleftrightarrow &SE_{1,4}(r), \\
{} & B_{1,3}(r)& \longleftrightarrow &SE_{1,6}(r), \\
{} & B_{1,1}(r)& \longleftrightarrow &SE_{1,2}(r). \\
\end{array}
$$
By Lemma~\ref{lem:spcom} (item $(ii)$ (1) and item $(i)$ (3)), we have the following equalities: for any $r\in A$,
$$
B_{1,3}(r)=[D_{1,2}(1),B_{2,3}(r)],\ \ 
B_{1,1}(r)=B_{1,3}(-r)[B_{3,3}(-r),D_{1,3}(1)].
$$
These equalities together with the inequalities in the paragraph above imply that for any $\xi \in b_1(N_1)$ and any $r\in A$,
$$
\|\xi -\rho(B_{1,3}(r))\xi\| \leq 64\mathcal{K}^{-1}C, \textrm{ and }\|\xi -\rho(B_{1,1}(r))\xi\| \leq 128\mathcal{K}^{-1}C.
$$
In a similar way, we also have $\|\xi -\rho(B_{1,2}(r))\xi\| \leq 64\mathcal{K}^{-1}C$ in the setting above. Therefore we conclude the following:
$$
\textrm{for any $\xi\in b_1(N_1)$ and any $n \in N$, }
\|\xi -\rho(n)\xi\| \leq 304\mathcal{K}^{-1}C.
$$
Note that the right hand side is independent of the choices of $\xi \in b_1(N_1)$ and $n \in N$. By Lemma~\ref{lem:cheb} 
and the trivial fact that $(B_{\rho(N)}')^{\rho(N)}=0$, this inequality forces that $
\sup_{\xi \in b_1(N_1)}\|\xi\| \leq 304 \mathcal{K}^{-1}C$. 
This means that $b_1(N_1)$ is bounded. Thus we have shown assertion $(\mathrm{A}1)$. 

Finally, we will confirm assertion $(\mathrm{A}2)$. Define a \textit{finite} subset $S_2$ of $G$ as follows:
\begin{align*}
S_2&:=\{ B_{1,2}(\pm x_l), B_{1,3}(\pm x_l), B_{2,3}(\pm x_l), D_{1,2}(\pm x_l), D_{2,1}(\pm x_l), D_{1,3}(\pm x_l), D_{2,3}(\pm x_l):  0\leq l\leq k\} \\
&\ \ \cup \{B_{i,i}(\pm x_1^{\epsilon_1}\cdots x_k^{\epsilon_k}):1\leq i\leq 3, \epsilon_1 ,\ldots ,\epsilon_k\in \{0,1\} \} \\
&= \left\{ \left(
\begin{array}{ccc|ccc} 
 1 & * & * & * & * & * \\
 * & 1 & * & * & * & * \\
 0 & 0 & 1 & * & * & * \\
 \hline 
 {} & {} & {} & 1 & * & 0 \\
 {} & 0 &  {} & * & 1 & 0 \\
 {} & {} & {} & * & * & 1 
 \end{array}\right) \right\} \cap S.
\end{align*}
Then the following holds true:
\begin{align*}
\textrm{for any $s\in S_2$ and $\gamma\in N_2$,}\textrm{ there exists $l\in N_1$ such that\ \ }  s\gamma=\gamma s l.
\end{align*}
More precisely, $l=e$ unless $s$ is of the form of $D_{1,2}$ or $D_{2,1}$. Since we have already verified assertion $(\mathrm{A}1)$, we know that $b_1(N_1)$ is bounded. Therefore 
$$
C':=\max \{ \sup_{s\in S_2}\| b_1(s)\|, \sup_{g,h\in G}\| b_1(gh)-b_1(g)-\rho(g)b_1(h)\|, \sup_{l\in N_1}\| b_1(l)\| \}
$$
is a finite real number. 
We have the following inequalities for any $\gamma\in N_2$ and $s\in S_2$:
\begin{align*}
  &\| b_1(\gamma)-\rho(s)b_1(\gamma)\| \leq \| b_1(\gamma) - b_1(s\gamma)\| +\|b_1(s)\| +C' \\
                            \leq &\| b_1(\gamma) - b_1(s\gamma)\| +2C' 
                            = \| b_1(\gamma) - b_1(\gamma sl)\| +2C' \\
                            \leq &\| b_1(\gamma)-b_1(\gamma) -\rho(\gamma)b_1(s)- \rho(\gamma s)b_1(l)\| +4C' \leq 6C'.
\end{align*}
Here $l\in N_1$ is chosen such that $s\gamma=\gamma sl$ as in above. Note that the very below term in these inequalities is independent of the choice of $\gamma \in N_2$. Therefore we can verify assertion $(\mathrm{A}2)$ in a similar way to one in the proof of assertion $(\mathrm{A}1)$.

We have verified assertions $(\mathrm{A}1)$ and $(\mathrm{A}2)$, and thus have completed the proof of the theorem.

\end{proof}

\begin{proof}(Theorem~\ref{thm:FFLp})
The conclusion follows from Proposition~\ref{prop:TtoFFsp} and Theorem~\ref{thm:short}, in view of Theorem~\ref{thm:ex} (recall from $(ii)$ of Theorem~\ref{thm:BFGM1} that relative property $(\mathrm{T}_{\mathcal{L}_p})$ is equivalent to relative property $(\mathrm{T})$!). Here again we utilize property $(*)$ in Subsection~\ref{subsec:TTmodT}, see Theorem~\ref{thm:bddgen}.
\end{proof}

\section{Shortcut of Theorem~\ref{mthm:MCGOut1} for universal lattices}\label{sec:short}
Here we see short and much less involved proof of Theorem~\ref{mthm:MCGOut2}. The key is study on distorted elements in a group. We prove Theorem~\ref{mthm:MCGOut2}. 
\begin{defn}\label{def:dist}
Let $G$ be a finitely generated group. An element $g\in G$ is called a \textit{distorted element} if $\lim_{n\to \infty }|g^n|_S/n=0$.
Here $S$ is a finite generating set (the choice of such $S$ does not affect the definition above), and $|\cdot|_S$ denotes the word length with respect to $S$. The element $g$ is said to be \textit{undistorted} otherwise.
\end{defn}

Note that in the literatures, it is more common to ask a distorted element \textit{not} to be a torsion. Here we allow torsions as well.

The following theorems respectively Farb--Lubotzky--Minsky \cite{FLM} (and L. Mosher); and E. Alibegovi\'{c} \cite{Ali} state torsions are only distorted elements in $\mathrm{MCG}(\Sigma)$; $\mathrm{Out}(F_n)$.

\begin{thm}$($\cite{FLM},\cite{Mos};\cite{Ali}$)$\label{thm:undist}
Let $\Sigma=\Sigma_{g.l}$ be a surface and $n\geq 2$. Then any element respectively in $\mathrm{MCG}(\Sigma)$; and in $\mathrm{Out}(F_n)$ which is not a torsion is undistorted.
\end{thm}
This result together with the following also well-known fact is the key in this section.
\begin{prop}\label{prop:torsi}
Groups $\mathrm{MCG}(\Sigma)$ for a non-exceptional surface and $\mathrm{Out}(F_n)$ are virtually torsion-free. Namely, there exists finite index torsion-free subgroups.
\end{prop}
Before proceeding to the proof of Theorem~\ref{mthm:MCGOut2}, we recall the definition of the \textit{Steinberg group}. The Steinberg group over a unital ring $R$ of degree $m$ is defined as follows: it is the group generated by $\{x_{i,j}(r) :1\leq i,j\leq m, i\ne j\}$ (the set of formal generators) which subjects to the following commutator relations:
\begin{align*}
&x_{i,j}(r)x_{i,j}(s)=x_{i,j}(r+s), & \\
&[x_{i,j}(r),x_{k,l}(s)]=e   &\textrm{if }i\ne l,j\ne k, \\
&[x_{i,j}(r),x_{j,k}(s)]=x_{i,k}(rs)   &\textrm{if }i\ne k. 
\end{align*}
The Steinberg group is written as $\mathrm{St}_m(R)$, and there is a natural surjection $\mathrm{St}_m(R)  \twoheadrightarrow \mathrm{E}_m(R)$, which sends each $x_{i,j}(r)$ to $E_{i,j}(r)$.

\begin{proof}(Theorem~\ref{mthm:MCGOut2})
Here we only argue the case of $H=\mathrm{MCG}(\Sigma)$ targets. 
Set $R=\mathbb{Z}\langle x_1,\ldots ,x_k\rangle$. By Proposition~\ref{prop:torsi}, there exists a finite index subgroup of $H$ which is torsion-free. We choose one and name it $H_0$.

 Let $\Psi \colon \Gamma \to H$ be a homomorphism. Then $\Phi(\Gamma)\cap H_0$ is a finite index subgroup in $\Phi (\Gamma)$ and hence $\Gamma_0:= \Phi^{-1}(\Psi(\Gamma)\cap H_0)$ is a finite index subgroup of $\Gamma$. For distinct pair $(i,j)$ with $1\leq i,j \leq m$, define a subset of $R$ as $R^0_{i,j}:=$$\{r\in R : E_{i,j}(r)\in \Gamma_0 \}$. Since for fixed $i,j$ $\{E_{i,j}(r):r\in R\} $$\cong R$ as additive groups, $R^0_{i,j}$ is a finite index subgroup in the additive group $R$. 

Note that if $E_{i,j}(r)\in \Gamma$, then it is a distorted element in $\Gamma$. This follows form the commutator relation:
 $$
 [E_{i,j}(r),E_{j,l}(s)]=E_{i,l}(rs) \ \ \ \ (i\ne j ,j\ne l, l\ne i;r,s\in R).
 $$

 By the construction of $H_0$, we conclude that for any $(i,j)$ and any $r\in R^0_{i,j}$, $\Phi(E_{i,j}(r))=e_H$. Indeed, the relation above implies each $\Phi(E_{i,j}(r))$ ($r\in R^0_{i,j}$) is distorted, and Theorem~\ref{thm:undist} shows that only distorted element in $H_0$ is $e_H$. We next claim that $R^0:=\bigcap_{i,j}R^0_{i,j}$ is a \textit{subring} of $R$. Indeed, the commutator relation above implies is $R$ is closed under multiplication. 
Therefore $R^0$ is a subring of $R$ of finite index. We use the following theorem of J. Lewin (Lemma $1$ in \cite{Lew}): \textit{for any finitely generated $($possibly noncommutative$)$ ring $Q$, any finite index subring $Q_0$ of $Q$ contains a $($two sided$)$ ideal $J$ of $Q$ which is a finite index subring of $Q_0$}. Thus we have a ideal $\mathcal{I}$ of $R$ of finite index, which is included in in $R_0$. 

Finally, we employ the following folklore result (the proof can be found in Lemma 17 of \cite{KaSa}): \textit{for a finite ring} $S$ \textit{the Steinberg group} $\mathrm{St}_{m\geq 3}(S)$ \textit{is finite.} Now we are done because $\Phi$ factors through  a subgroup of $\mathrm{St}_m(R/\mathcal{I})$, and this group itself is finite. 
\end{proof}

\begin{rem}\label{rem:spdame}
Here is a remark on why this argument does not work for symplectic universal lattice cases. The finiteness of the Steinberg group over a finite ring is extended to general cases, more precisely, twisted Steinberg groups (see Proposition 4.5 of a paper \cite{Rap} of I. A. Rapinchuk). However, as we have seen in Lemma~\ref{lem:spcom}, the commutator relation among elementary symplectic groups are much complicated. The gap lies in the point that we have deduced $R_0$ in the proof above is multiplication closed.
\end{rem}

\end{document}